\documentclass[10pt]{article}

\usepackage{amssymb,amsmath,amsfonts,amsthm,enumerate,color,mathrsfs,hyperref,amssymb,extarrows}
\usepackage[shortlabels]{enumitem}
\usepackage[toc,page,title,titletoc,header]{appendix}
\usepackage{graphicx}
\usepackage{indentfirst}
\usepackage{multicol}
\usepackage{booktabs}
\usepackage{appendix}
\usepackage{mathtools}
\DeclarePairedDelimiterX{\norm}[1]{\lVert}{\rVert}{#1}
\newtheorem{theorem}{Theorem}[section]
\newtheorem{lemma}[theorem]{Lemma}
\newtheorem{proposition}[theorem]{Proposition}

\newtheorem{remark}[theorem]{Remark}
\numberwithin{equation}{section}

\newcommand{\dif}[1]{\left[ #1 \right]}

\def\to{\rightarrow}

\def\q{\quad}

\def\l{\langle}
\def\r{\rangle}
\def\no{\noindent}
\def\nb{\nonumber}

\newcommand{\hh}[1]{{\color{blue}{#1}}}

\def \sss{\scriptscriptstyle}

\def\bb{\begin{equation}
  \left\{
   \begin{array}{l} }
\def\ee{   \end{array}
  \right.
  \end{equation}}

 \def\beqn{\begin{eqnarray}}  \def\eqn{\end{eqnarray}}
\def\beqnx{\begin{eqnarray*}} \def\eqnx{\end{eqnarray*}}

\def\mm{ \left[
 \begin{matrix}}
\def\nn{\end{matrix} \right] }

\def\p{\partial}
\def \dd{\cdot}
\def \t{\times}

\newcommand{\gk}[1]{G_{#1,\lat}}

\def \gg{\nabla}
\def\curl{\nabla \times}

\def\de{\Delta}

\def \cs{{\rm curl}_{\p B}}
\def \csv{\vec{{\rm curl}}_{\p B}}
\def \lap{\Delta_{\p B}}
\def \divs{\nabla_{\p B} \dd}
\def \gs{\nabla_{\p B}}
\def \cst{{\rm curl}_{\Gamma}}
\def \cvt{\vec{{\rm curl}}_{\Gamma}}

\def\n{\nu}
\def\nt{\nu \times}

\newcommand{\pn}[1]{\frac{\partial #1}{\partial \n}}

\newcommand{\sm}[1]{\sum_{n = #1}^\infty}

\def \w {\widetilde}
\def \b{\overline}
\def \h{\hat}
\def \c{\check}

\def \R{\mathbb{R}}

\def \G{\mathbb{G}}
\def \C{\mathbb{C}}
\def \hh{\mathbb{H}}

\def \I{\mathbb{I}}
\def \ts{\mathrm{T}}
\def\d{\delta}
\def \vp{\varphi}
\def\k{{\rm \bf k}}
\def\kb{k'}
\def\xb{x'}

\def\kd{k_3}
\def\xd{x_3}

\def \emm{{\rm e}/{\rm m}}
\def \me{{\rm m}/{\rm e}}
\def \zm{{\rm m}}
\def \ze{{\rm e}}
\def \e{{\rm \bf e}}
\def \a{{\rm \bf a}}
\def \b{{\rm \bf b}}
\def \x{{\rm \bf x}}
\def \y{{\rm \bf y}}
\def \A{\mathcal{A}}
\def \T{\mathcal{T}}
\def \B{\mathcal{B}}
\def \W{\mathcal{W}}

\def \H{\mathcal{H}}

\def \Hs{\mathcal{H}^*}

\def \po{{\rm \bf p}}
\def \di{{\rm \bf d}}
\def \lat{\#}

\def\ei{E^i}

\def\hi{H^i}

\def\ep{electric permittivity\ }
\def\mmp{magnetic permeability\ }
\def\eec{\varepsilon_c}
\def\muc{\mu_c}
\def\eem{\varepsilon}
\def\mum{\mu}

\def\im{\mathfrak{Im}}
\def \lm {\lambda_\mu}
\def \lee {\lambda_\varepsilon}

\def\D{\mathcal{D}}

\def \kk{\mathcal{K}_{\emm}}
\def \np {\mathcal{K}_{\emm}^*}

\def \so{\mathcal{S}_{\emm}^0}
\def \no {\mathcal{K}_{\emm}^{0,*}}
\def \ko{\mathcal{K}_{\emm}^0}

\newcommand{\app}[1]{\mathcal{A}_{p,#1,0}}

\def \L{\mathcal{L}}
\def \M{\mathcal{M}}
\def \S{\mathcal{S}}
\def \K{\mathcal{K}}
\def \T{\mathcal{T}}

\def \htt{H_T^{\frac{1}{2}}(\p B)}
\def \hst{H^{\frac{1}{2}}(\p B)}
\def \hsn{H^{-\frac{1}{2}}(\p B)}
\def \htn{H_T^{-\frac{1}{2}}(\p B)}
\def \htd{H_T^{-\frac{1}{2}}({\rm div},\p B)}
\def \htc{H_T^{-\frac{1}{2}}({\rm curl},\p B)}

\newcommand{\pb}[1]{\w{\phi}_{0,#1}}
\newcommand{\psb}[1]{\w{\psi}_{0,#1}}

\title{Mathematical analysis of electromagnetic plasmonic metasurfaces}
\begin{document}
\author{
Habib Ammari\footnote{Department of Mathematics, ETH Z\"{u}rich, R\"{a}mistrasse 101, CH-8092 Z\"{u}rich, Switzerland
(habib.ammari@math.ethz.ch).}
\and Bowen Li\footnote{Department of Mathematics, The Chinese University of Hong Kong, Shatin, N.T., Hong Kong. (bwli@math.cuhk.edu.hk).}
\and Jun Zou\footnote{Department of Mathematics, The Chinese University of Hong Kong, Shatin, N.T., Hong Kong.
The work of this author was
substantially supported by Hong Kong RGC grant (Project 14322516).
(zou@math.cuhk.edu.hk).}
}
\date{}
\maketitle
\begin{abstract}
We study the anomalous electromagnetic scattering in the homogenization regime, by a subwavelength thin layer of periodically distributed plasmonic nanoparticles on a perfect conducting plane.
By using layer potential techniques, we derive the asymptotic expansion of the electromagnetic field away from the thin layer and quantitatively analyze the field enhancement
due to the mixed collective plasmonic resonances, which can be characterized by the spectra of periodic Neumann-Poincar\'{e} type operators. Based on the
asymptotic behavior of the scattered field in the macroscopic scale, we further demonstrate that
the optical effect of this thin layer can be effectively approximated by a Leontovich boundary condition, which is uniformly valid no matter whether the incident frequency is near the resonant range but varies with
the magnetic property of the plasmonic nanoparticles.
The quantitative approximation clearly shows the blow-up of the field energy and the conversion of polarization
when resonance occurs,
resulting in a significant change of the reflection property of the conducting plane.
These results confirm essential physical changes of electromagnetic metasurface at resonances mathematically, whose occurrence was verified earlier for the acoustic case \cite{ammari2017bubble} and the transverse magnetic case \cite{ammari2016mathematical}.
\end{abstract}

\section{Introduction}
The study of electromagnetic scattering by a thin layer composed of periodic subwavelength resonators,  which can strongly interact with the incident wave, have received considerable attention recently for their possibilities of realizing the full control of reflected and transmitted waves \cite{chen2016review,tretyakov2015metasurfaces,zhang2016advances,huidobro2017tunable}.
Such thin layers of composite material, usually referred to as the ultrathin
metasurfaces in the physical and engineering community, have a macroscopic effect on the scattered wave
although the layer thickness, or the size of cell structure, is negligible with respect to the operating wavelength \cite{felbacq2013layer,felbacq2015impedance,kraft2015designing,kraft2016bianisotropy,bonnetier2010asymptotic,lin2018scattering,lin2018scatteringhomo}.
We refer the readers to \cite{sun2019electromagnetic} for a systematic review of
the electromagnetic metasurfaces and its applications.
Great effort has been made recently by the mathematical community to develop a universal theory
for a better understanding of the mechanism underlying the metasurfaces.
It turns out that these anomalous scattering phenomena have a close relation
with the multiscale nature of the subwavelength cell structures and the excitation of various resonances.
A systematic study
was carried out in \cite{lin2017scattering,lin2018scattering,lin2018scatteringhomo,lin2019integral, lin2019fano}
to understand the electromagnetic scattering by the perfect conducting slab patterned
with the subwavelength narrow slits under varying regimes and periodic patterns.
And it was shown in \cite{lipton2017novel} that the scattering effect by a novel metasurface made of periodically corrugated cylindrical waveguides can be approximated by smooth cylindrical waveguides with an effective metamaterial surface
impedance.

Plasmonic nanoparticles such as gold and silver are popular choices for the subwavelength resonators in the electromagnetic setting due to their unique optical properties \cite{maier2007plasmonics}, and even a thin layer of these
particles can significantly influence the wave propagation pattern.
In this work, we shall consider the scattering effect of a thin layer of periodical plasmonic nanoparticles of
subwavelength mounted on a perfectly conducting plane in the homogenization regime, i.e.,
the period of the structure is about the same size as the nanoparticles but is much smaller than the incident wavelength.
At the quasi-static limit, a single nanoparticle can exhibit the plasmonic resonances at some specific frequencies that are related to the spectra of the Neumann-Poincar\'{e} operators.
We refer to \cite{ammari2016surface, ammari2016plasmaxwell, ammari2017mathematicalscalar}
for the mathematical analysis of plasmonic resonances.
However, when we consider the homogenization regime, the mixed and collective plasmonic resonances may occur,
which are very different from the single plasmonic nanoparticle case in free space.
It is interesting to note that if the thin layer is made of normal dielectric materials
with biperiodic conducting inclusions covering a cylindrical body,
a Leontovich boundary condition can be derived to approximate the effect of the layer \cite{abboud1995diffraction}.
Such grating problems or boundary layer effects have been extensively studied by matched asymptotic
expansion techniques, see, e.g., \cite{bensoussan2011asymptotic,achdou1998effective,abboud1996diffraction, allaire1999boundary,delourme2015high,delourme2013well}.
However, as we shall characterize, the cell problem here is nearly singular at some frequencies
if the nanoparticles are plasmonic. And in this case, the standard homogenization is not applicable and
the reflection coefficients may blow up. Therefore, we should seek new analytical tools for deriving
the exact blow-up order and justifying the validity of the approximation of the Leontovich-type boundary conditions.
In the present work, we use  layer potential techniques to study the reflection properties of electromagnetic plasmonic metasurfaces,
which is more general than the framework recently proposed in \cite{ammari2016mathematical}. The technique was used in \cite{ammari2017bubble}
to illustrate the superabsorption of acoustic waves with bubble metascreens observed in \cite{leroy2015superabsorption}.

As we shall point out in Section \ref{concluding},
our results and analyses in this work apply to several important physical regimes and applications,
in particular, to the general physical setting that involves several physical scales, namely,
the distance between every two thin layers of periodically distributed plasmonic nanoparticles,
the incident wavelength, sizes of nanoparticles, and
the period of each layer of periodical nanoparticles, can be of very different multiscale,
such as
\begin{eqnarray*}
&&\mbox{size of particle $\ll$ period $\ll$ distance $\sim$ wave length, or} \\
&&\mbox{size of particle $\ll$ period $\sim$ distance $\ll$  wave length.}
\end{eqnarray*}

%
%
%
The paper is organized as follows. In the next section, we describe our model mathematically and introduce some notation and definitions. In Section \ref{layerpotential}, we first introduce the quasi-periodic layer potentials and derive the corresponding asymptotic expansions, and then recall some basic results concerning the Neumann-Poincar\'{e} operators and establish the resolvent estimates for the leading-order potentials.  The Section \ref{sec:Far-field} is the main contribution of this work, devoted to the calculation of the far-field asymptotic expansion of the scattered wave and
a boundary condition approximation under the excitation of plasmons.
We shall end our work with some concluding and extension remarks.

\section{Problem descriptions and preliminaries} \label{setup}
This section is devoted to the basic setup and the mathematical formulation of the electromagnetic scattering problem. We shall write $\R^3 \ni \x = (x_1, x_2, x_3) = (x' ,x_3)$ with $x' = (x_1, x_2) \in \R^2 $ and $ x_3 \in \R$,
$\Gamma :=\{x \in \R^3 | x_3 = 0 \}$ for the reflective plane and $\R^3_{\pm}:=\{\x \in \R^3 | \pm x_3 > 0 \} $
for the upper and lower half spaces. We denote by $(\e_1,\e_2,\e_3)$ the usual Cartesian basis of $\R^3$. For a multi-index $\alpha \in \mathbb{N}^3$, we write $\x^\alpha = x_1^{\alpha_1} x_2^{\alpha_2} x_3^{\alpha_3}$ and $\p^\alpha = \p_1^{\alpha_1} \p_2^{\alpha_2} \p_3^{\alpha_3}$ with $\p_j = \frac{\p}{\p x_j}$. We shall always use
$B\in \mathbb{R}^3_+$ to denote a $C^2$ smooth bounded domain with its size of order one, and
use $D:=\d B$ to describe a single nanoparticle and $\D$ for
the collection of plasmonic nanoparticles periodically distributed along
a lattice $\Lambda^\d$ given by
\begin{equation*}
    \Lambda^\d = \{{\rm R}^\d \in \R^2; {\rm R}^\d = n_1 \d \textbf{a}_1 + n_2 \d \textbf{a}_2, n_i \in \mathbb{Z}\},
\end{equation*}
in which $\textbf{a}_1, \textbf{a}_2$ are linearly independent vectors lying in $\Gamma$ with $|\a_1| \sim |\a_2| \sim 1$.
Then we can write $\D = \bigcup_{{\rm R}^\d \in \Lambda^\d}(D +{\rm R}^\d)$.
For convenience, we shall write $\Lambda^1$ as $\Lambda$,
and we can see $\D = \bigcup_{{\rm R}\in \Lambda}\d(B + {\rm R})$.
We now define the  cell $\Sigma$, $\Omega$ in $\Gamma$ and in $\R^3_+$ respectively by
\begin{align*}
    &\Sigma =  \Big\{ \textbf{a} \in \R^2; \textbf{a} = c_1 \textbf{a}_1 + c_2 \textbf{a}_2, c_i \in (-\frac{1}{2},\frac{1}{2})
    \Big\}, \\
    &\Omega = \Big\{\textbf{a} \in \R^2; \textbf{a} = c_1 \textbf{a}_1 + c_2 \textbf{a}_2, c_i \in (-\frac{1}{2},\frac{1}{2})
    \Big\}\t (0,\infty).
\end{align*}
We further assume that $B$ is contained in $\Omega$ with the distance from the reflective plane $\Gamma$ of order one and the dimensionless quantity $\d$ is much less than one, since we are interested in the homogenization regime. For any $\w{\x} \in \p B$, we have $\x = \d \w{\x} \in \p D$. Then for a function $\varphi(\x)$ defined on $\p D$,
its pull back $\w{\varphi}(\w{\x}) := \varphi(\d \w{\x}) =\varphi(\x)$ is defined on $\p B$, and this convention is adopted
throughout this work. In particular, if we denote by $\n(\x)$ the exterior normal vector of $\p D$, then its pull back  $\w{\n}(\w{\x})$ is the normal vector of $\p B$. But we may also simply write $\n$ for a normal vector without specifying its definition domain when no confusion is caused. For the sake of exposition, we often refer to $\w{\x}$, $B$ and $\Omega$  as the reference variable, reference domain and reference cell, respectively.

We shall consider the \ep $\eec(\omega)$ and \mmp $\muc(\omega)$ of the nanoparticle are described by the Drude model \cite{ammari2016surface,ammari2017mathematicalscalar,maier2007plasmonics}.
Although explicit formulas for $\mu_c$ and $\eec$ are available in terms of the Drude model,
it suffices for all our analysis and arguments
to generally assume that both $\mu_c$ and $\eec$ are complex numbers with $\im \mu_c, \im \eec \ge 0$, and depend on the frequency $\omega$ of the incident wave.
We write the permittivity and permeability of the background medium by $\eem $ and $\mum $, and further assume them to be constant $1$
after an appropriate scaling.
Then the wave number $k_c(\omega)$ and $k$ are given by
$$
k_c(\omega) = \omega\sqrt{\eec(\omega)\muc(\omega)} \quad \text{and} \quad k = \omega\sqrt{\eem\mum} = \omega.
$$

We are now ready to formulate the scattering problem of our interest as follows:
\begin{equation}\label{model}
  \left\{  \begin{array}{ll}
  \curl E = i k \mu_\D H       & \text{in} \q \R_+^3\backslash \partial \D,  \\
  \curl H = - i k \varepsilon_\D E       & \text{in} \q \R_+^3\backslash \partial \D, \\
 \left[\n \t E \right]= \dif{\n \t H}= 0 &  \text{on} \q \p \D,\\
 \e_3 \t E = 0 \q & \text{on} \q \Gamma,
    \end{array} \right.
\end{equation}
where $E-\ei$ and $H-\hi$ satisfy certain outgoing
radiation conditions,  $\varepsilon_{\D} := \mathcal{X}(\R^3_+\backslash  \D) + \eec \mathcal{X}(\D)$, $\mu_{\D} := \mathcal{X}(\R^3_+\backslash \D) + \muc \mathcal{X}(\D)$, with $\mathcal{X}$ being the standard characteristic function.
Throughout the work, we use $[\dd]:= \dd|_- - \dd|_+$ to denote the jump across the interface
$\p \D$, and the subscripts $\pm$ to denote the limits taken from the outside and inside of $\D$ respectively.
The incident plane wave $(\ei, \hi)$ is given by
$$ 
\ei = \po e^{i k \di \cdot \x} - \po^* e^{i k \di^* \cdot \x}\,, \q
\hi = \di \t \po e^{i k \di \cdot \x} - \di^* \t \po^* e^{i k \di^* \cdot \x},
$$ 
where $\di$ is a unit vector for the incident direction with $d_3<0$, and $\po$ is the polarization direction. Here and in the sequel we often use the superscript $*$ to denote the reflection of a vector with respect to $\Gamma$, i.e., $\di^* = (d',-d_3)$.
But the notation $*$ may have other meanings at different occasions, so we will illustrate the actual meaning of $*$ whenever it may cause confusion.
Denote by $\k = k \di,\, \k^* = k \di^*$ the wave vector and its reflection respectively. We are  interested in finding a quasi-periodic solution $(E,H)$ to the system
\eqref{model} such that
$$
E(\x+ {\rm R}^\d) = e^{i\k \cdot  {\rm R}^\d} E(\x), \q
H(\x+ {\rm R}^\d) = e^{i\k \cdot  {\rm R}^\d} H(\x).
$$
Hence we have the usual Rayleigh Bloch expansion for the scattered field in the domain above the layer of nanoparticles.
As in \cite{ammari2017bubble}, we impose the outgoing radiation condition on
the solutions to the system \eqref{model} by assuming that all the modes in the Rayleigh-Bloch expansion are either decaying exponentially or propagating along the $x_3$-direction.
Under the subwavelength assumption, the period of lattice is of order $\d$, and the scattered wave consists of only a single propagative mode in the far field, namely
\begin{equation*}
    E^r:= E- E^i \sim \po^r e^{ik'\dd x'}e^{-ik_3 x_3} \q \text{as} \ x_3 \to \infty \, \mbox{\quad
    (for some polarization direction $\po^r$)}.
\end{equation*}

The remaining part of this section is devoted to introduce more notation, definitions and recalling some basic results concerning the surface differential operators and function spaces that are frequently used in the sequel.
For $s \in \R$, we denote by $H^s(\p B)$ and $H_T^s(\p B)$
the usual Sobolev space of order $s$ of scalar functions and tangential vector fields on $\p B$, respectively,
and denote by $H_0^s(\p B)$ the zero mean subspace of $H^s(\p B)$.
Also, the Sobolev spaces $H^s(B)$ and $H^s_{loc}(\Omega\backslash B)$
are needed, as well as the trace operator $\gamma_0: H^s(B) \to H^{s-\frac{1}{2}}(\p B)$ for
$s > \frac{1}{2}$.
We introduce the surface gradient $\gs$ and the surface vector curl (written as
$\csv$) in the standard way \cite{nedelec2001acoustic}, which map $\hst$ to $\htn$. Their corresponding adjoint operators are the surface divergence $\divs$ and the surface scalar curl, i.e., $\cs$: $\htt \to \hsn$. And it holds in $\hsn$ that
\begin{equation}
    ker(\gs) = ker(\csv) = \R.
\end{equation}
The Laplace-Beltrami operator $\lap := \divs \gs = - \cs \csv$ shall also be used.
For a vector field $u \in \htn$, we will often need its tangential component $r(u) := \n \t u$. 
It is easy to check by using definition and duality relation that each of the following identities holds for a suitable function $\varphi$,
\begin{align}
&\csv \varphi =  -r(\gs \varphi), \q
\cs \varphi =  -\divs(r\varphi),\label{bapro:sur:1} \\
& \divs \csv \varphi = 0, \q \cs \gs \varphi = 0. \label{bapro:sur:2}
\end{align}
Moreover, we introduce the spaces $H({\rm curl}, B)$,  $H_{loc}({\rm curl}, \Omega\backslash \bar{B})$, $H({\rm div}, B)$ and $H_{loc}({\rm
div}, \Omega\backslash \bar{B})$ of (locally) square integrable vector fields with (locally) square integrable curl and divergence, respectively.
We will frequently use the normal trace $\gamma_n(u):= u \dd \n|_{\p B}$, the tangential trace $\gamma_t(u):= \n \t u|_{\p B}$ and the tangential component trace
$\pi_t(u):= (\n \t u)\t \n|_{\p B}$ for appropriately smooth vector fields $u$. Indeed,
$\gamma_n,\gamma_t$ and $\pi_t$ can be extended to linear continuous mappings from $H({\rm div},B)$ to $\hsn$, $H({\rm curl},B)$ to $\htd$, and $H({\rm curl},B)$ to $\htc$ respectively, where
\begin{align*}
&\htd = \{\varphi \in \htn;\q \divs \varphi \in H^{-\frac{1}{2}}(\p B)\},
\\
&\htc = \{\varphi \in \htn;\q \cs \varphi \in H^{-\frac{1}{2}}(\p B)\}.
\end{align*}
It is known that $\htc$ can be identified with the dual space of $\htd$ with duality pairing $\l \psi,\varphi \r: = \int_{\p B} \psi \dd \varphi d\sigma $ for smooth vector fields $\psi,\varphi$ (cf.\cite{monk2003finite,buffa2002traces}).
And for $f \in H^1(B)$, we have
\begin{equation} \label{relation:trace:l1}
    \gs \gamma_0(f) = \pi_t(\gg f).
\end{equation}
Similarly, it holds for $u \in H({\rm curl},B)$,
\begin{equation}
\label{relation:trace:curl}
\cs \pi_t(u) = \gamma_n(\curl u).
\end{equation}
For our subsequent analysis, the Helmholtz decomposition of $\htd$ is frequently used (cf.\cite{buffa2002traces}):
\begin{equation*}
\htd = \gs H^{\frac{3}{2}}_0(\p B) \oplus \csv H^{\frac{1}{2}}_0(\p B)\,.
\end{equation*}

In this work, we denote by $\otimes$ the tensor product operation of two vectors, i.e., given two vectors $\a \in \R^n$ and $\b \in \R^m$, $\a \otimes \b$ is a $n \times m$ matrix given by $(\a \otimes \b)_{ij} = a_i b_j$, and let vector operators act on matrices column by column.  For any two
Banach spaces $X$ and $Y$, we write by $\L(X,Y)$ the set of all linear continuous mappings from $X$ to $Y$, or simply by $\L(X)$ if $Y = X$.
We write $\norm{\dd}_X$ for the norm defined on the space $X$ and $_{X^*}\l\dd,\dd\r_{X}$ for the natural duality pairing between $X$ and its dual space $X^*$. However, we may simply write $\norm{\dd}$ and $\l\dd,\dd \r$
without specifying the subscripts when no confusion is caused.
We will not identify the dual spaces of Hilbert spaces with themselves,
instead we always regard them as the subspaces of distributions.
Hence all the adjoint operators in this work are introduced by their natural duality pairings.
 We end this section by introducing the expression $x \lesssim y$, which means $x\leq C y$ for some generic constant $C$. If $x \gtrsim  y$ and $x \lesssim y$ holds simultaneously, then we write $x \approx y$.


\section{Layer potential techniques} \label{layerpotential}
Before considering the scattering problem, we present some preliminary knowledge on the quasi-periodic layer potential techniques in this section for our subsequent analysis. We first introduce the quasi-periodic Green's tensors satisfying certain boundary conditions and compute their asymptotic expansions with respect to $\d$. Then we study the associated layer potentials, as well as their asymptotics. After that, we turn our attention to the properties of the leading-order potentials and the resolvent estimates for Neumann-Poincar\'{e}-type operators. These results will be the foundation for the far-field asymptotics and approximation error estimate conducted in the next section.

\subsection{Quasi-periodic Green's tensors and basic properties}
Following the notation in \cite{ammari2017bubble}, we start with  the scalar quasi-periodic Green's function
$G_\lat^\k$ with complex wavenumber $k$ with $\im k \ge 0$, which is the solution to
\begin{equation}
(\de + k^2)G_\lat^\k(\x) = \sum_{{\rm R} \in \Lambda}e^{i\kb \dd \xb}\d_{{\rm R}}(\x)= \sum_{{\rm R} \in \Lambda}e^{i \kb \dd {\rm R}}\d_{{\rm R}}(\x)\,,
\end{equation}
satisfying a certain outgoing condition. In the distribution sense, $G_\lat^\k$ is well-defined and given by
\begin{equation}
    G^\k_\lat(\x) = \sum_{{\rm R} \in \Lambda}e^{i \kb \dd {\rm R}} G^k(\x,{\rm R})\,,
\end{equation}
where $G^k(\x,\y): =  - \frac{e^{ik|\x-\y|}}{4\pi|\x-\y|}$ is  the fundamental solution to the Helmholtz operator $\de + k^2$ in free space.
We further define $G_\lat^\k(\x,\y) := G_\lat^\k(\x-\y)$. For our purpose, we are interested in the behavior of the quasi-periodic Green's function  $G_{\lat^\d}^\k(\x)$ with respect to the lattice $\Lambda^\d$. With the reference variable $\w{\x}$, we easily observe that
\begin{equation} \label{for:scal:greenorg}
  G_{\lat^\d}^\k(\x) = \sum_{{\rm R}^\d \in \Lambda^\d} e^{i\kb \dd {\rm R}^\d}G^k(\x,{\rm R}^\d)  = \frac{1}{\d} \sum_{{\rm R} \in \Lambda} e^{i\d\kb \dd {\rm R}}G^{\d k}(\w{\x},{\rm R})\,.
\end{equation}
We thus have the following useful scaling property:
\begin{equation}    \label{for:scal:green}
   G_{\lat^\d}^\k(\x) = \frac{1}{\d}G_{\lat}^{\d \k}(\w{\x})\,.
\end{equation}
 Let $\Lambda^*$ be the reciprocal lattice of $\Lambda$ (cf.\cite{linton2010lattice}), and
$\tau$ be the volume of the unit cell of $\Lambda$. Then the explicit representation formula of $G^\k_\lat$ in the homogenization regime, i.e.,  $|k|  \ll \tau \sim 1$, is available \cite{ammari2017bubble}, as stated in the next theorem.
\begin{theorem}
Let $k\in \C$ be the complex wave number with $\im k  \ge 0 $. Assume that $|k|$ is small enough, then the quasi-periodic Green's function $G_\lat^\k$ can be expressed by
\begin{equation} \label{quasi-green}
 G_\lat^\k(\x) =\frac{i}{2 \tau \kd} e^{i\kb\dd \xb-i\kd |\xd|} - \frac{1}{2 \tau} \sum_{\xi \in \Lambda^* \backslash \{0\}} \frac{1}{\sqrt{|\xi+ k'|^2 -k ^2}}e^{i(\xi + k') \dd \xb}e^{-\sqrt{|\xi+ k'|^2 -k ^2} |x_3|}\,,
\end{equation}
where $\sqrt{z}$ is viewed as an analytic function defined by $\sqrt{z} = |z|^{1/2}e^{i\arg{z}/2}$ for $z \in \C \backslash \{-it, t \ge 0\}$.

In particular, when $k = 0$,
\begin{equation} \label{quasi-green:k=0}
    G^0_\lat(\x) = \frac{|\xd|}{2 \tau} - \frac{1}{2 \tau} \sum_{\xi \in \Lambda^* \backslash \{0\}} \frac{1}{|\xi|} e^{i \xi \dd \xb}e^{-|\xi|\dd|x_3|}.
\end{equation}
\end{theorem}
One can readily observe the symmetry property of $G^0_\lat$ from its representation formula \eqref{quasi-green:k=0},
\begin{equation} \label{recipro:grenn:k=0}
    G^0_\lat(\pm x',\pm x_3) = G^ 0_\lat(x',x_3)\,.
\end{equation}
Concerning the asymptotics of $G_\lat^{\d \k}$ with respect to $\d$, a direct application of Taylor's series gives us
\begin{equation}\label{expansion:green}
G_\lat^{\d \k}(\x) = \frac{i}{2\d \kd \tau} + G_{0,\lat}(\x) + \sum_{n =1}^\infty (\d k) ^nG_{n,\lat}(\x)\,.
\end{equation}
We remark that each term $G_{n,\lat}$ in \eqref{expansion:green} can be computed explicitly (cf.\cite{ammari2017bubble} for more details). In particular, for the leading-order term, we have
\begin{equation} \label{formula:exp:greenzero}
    G_{0,\lat}(\x) = \frac{k_3|\xd|-k' \dd x'}{2 k_3 \tau} - \frac{1}{2 \tau} \sum_{\xi \in \Lambda^* \backslash \{0\}} \frac{1}{|\xi|} e^{i \xi \dd \xb}e^{-|\xi|\dd|x_3|}  = G_\lat^0(\x) - \frac{\kb \dd x'}{2 \kd \tau}\,.
\end{equation}
Recalling from the definition of $G^{\d \k}_\lat$ that
$$
(\Delta + \d^2 k^2)G_{\lat}^{\d \k}(\x) = \sum_{{\rm R} \in \Lambda}e^{i \d \kb \dd \xb}\d_{{\rm R}}(\x)\,,
$$
we obtain, by substituting the expansion (\ref{expansion:green})  into the above formula,
\begin{align*}
& \Delta G_{0,\lat}(\x) + \d k (\Delta G_{1,\lat}(\x) + \frac{i}{2 d_3 \tau}) + \sum_{n=2}^\infty \d^n k^n (\Delta G_{n,\lat}(\x)+ G_{n-2,\lat}(\x)) \\
 = & \sum_{n =0}^\infty \d ^n k^n\sum_{{\rm R} \in \Lambda} \frac{(i d' \dd \xb)^n}{n!} \d_{{\rm R}}(\x)\,,
\end{align*}
which imply (with notation $G_{-1,\lat} = \frac{i}{2 d_3 \tau}$)
\begin{align}
\Delta \gk{0}(\x) = \sum_{{\rm R} \in \Lambda} \d_{{\rm R}}(\x), \q \text{and} \q
\Delta \gk{n}(\x) + \gk{n-2}(\x) = \sum_{{\rm R} \in \Lambda} \frac{(i d' \dd {\rm R})^n}{n!} \d_{{\rm R}}(\x)\,, \q n \ge 1\,. \label{for:recursive:original}
\end{align}

As the perfect conducting boundary condition is enforced only on the electric field,
we have to distinguish between the electric and magnetic Green's tensors in terms of
the boundary conditions. Their definitions rely on the quasi-periodic Green's functions with Dirichlet and Neumann boundary conditions, defined respectively by
\begin{align}
    G^\k_\ze(\x,\y) := G_{\lat}^\k(\x-\y) - G_\lat^\k(\x-\y^*)\,, \quad
    G^\k_\zm(\x,\y) := G_{\lat}^\k(\x-\y) + G_\lat^\k(\x - \y^*)\,. \label{def:dirichletgreen}
\end{align}
%
The asymptotics below follows directly from (\ref{expansion:green}),
\begin{equation}  \label{expansion:green:emm}
    G_{\emm}^{\d \k}(\x,\y) = \sum_{n = -1}^\infty (\d k)^n G_{n,\emm}(\x,\y)\,,
\end{equation}
where $G_{n,\emm}(\x,\y)$ are given by
\begin{equation} \label{eq:auxemeplocal}
     G_{n,\emm}(\x,\y) := \gk{n}(\x - \y) \mp \gk{n}(\x-\y^*)\q  \text{for}\  n \ge -1\,.
\end{equation}
Especially,
$G_{-1,\ze} = 0$ and  $G_{-1,\zm} = {i}/{(d_3 \tau)}$.
For the sake of exposition, here and in the sequel, we use the subscript $\emm$ to include
two cases, e.g., \eqref{expansion:green:emm} actually represents
two equations, obtained by replacing $\emm$ by e and m, respectively,
in \eqref{expansion:green:emm}.
Similarly, we shall also use $\me$ frequently.
Recalling (\ref{for:recursive:original}), if $\x,\y \in \bar{B}$, we have by noting that
$\d_{{\rm R}}(x'-y') = 0$ for all ${\rm R} \neq 0$,
\begin{align}
\Delta G_{0,\emm} = \d_{0}, \q and \q
\Delta G_{n,\emm} + G_{n-2,\emm} = 0\,,  \q n \ge 1. \label{for:Recursive relationship:green}
\end{align}
These recurrence relations shall be used in the calculation of asymptotic expansions of layer potential operators. According to the reciprocity (\ref{recipro:grenn:k=0}) of the periodic Green's function, we know
\begin{equation*}
    G_{\emm}^0(\x,\y) = G_{\emm}^0(\y,\x).
\end{equation*}
Combining this observation with  (\ref{formula:exp:greenzero}) and (\ref{eq:auxemeplocal}), we readily see the reciprocity is no longer suitable for $G_{0,\zm}(\x,\y)$ since
\begin{equation} \label{rela:leading-orderm}
    G_{0,\zm}(\x,\y) = G_\zm^0(\x,\y) - \frac{k'\dd (x'-y')}{k_3 \tau}\,,
\end{equation}
while $G_{0,\ze}(\x,\y)$ still behaves well due to
\begin{equation} \label{rela:leading-ordere}
    G_{0,\ze}(\x,\y) = G_\ze^0(\x,\y)\,.
\end{equation}
It is also worth mentioning that there is a singularity for $G^{\d \k}_\zm$ as $\d$ goes to $0$ (cf.\eqref{expansion:green} and \eqref{def:dirichletgreen}). This, together with non-symmetry of $G_{0,\zm}(\x,\y)$, makes
some of our subsequent analyses much more involved.
Finally, for our later use, we introduce the conjugate kernels $\hat{G}_{0,\emm}$ of $G_{0,\emm}$ by setting $\hat{G}_{0,\emm}(\x,\y):= G_{0,\emm}(\y,\x)$, namely,
\begin{align}
    & \hat{G}_{0,\ze}(\x,\y) = G_{0,\ze}(\y,\x) = G_\ze^0(\x,\y),
    \label{kernel:conjugate:e} \\
    & \hat{G}_{0,\zm}(\x,\y) = G_{0,\zm}(\y,\x) = G_\zm^0(\x,\y) +  \frac{k'\dd (x'- y')}{k_3 \tau}.
     \label{kernel:conjugate:m}
\end{align}

We are now ready to introduce the electromagnetic Green's tensor
\begin{equation} \label{def:tensor:em}
    \G_{\emm}^{\k}(\x,\y) = (1+ \frac{1}{k^2}\gg_\x\gg_\x \dd )\Pi^{\k}_{\emm}(\x,\y),
\end{equation}
where the matrix-valued functions $\Pi^{\k}_{\emm}$ are given by
\begin{equation} \label{def:kernel}
     \Pi^{\k}_{\emm}(\x,\y) =
     \mm
     G^\k_{\emm}\e_1,& G^\k_{\emm}\e_2,& G^\k_{\me}\e_3
     \nn(\x,\y)\,.
\end{equation}
It is easy to check that $\G_{\emm}^\k$ solve the equations
\begin{equation*}
    \nabla_\x \t \nabla_\x \t \G^\k_{\emm}(\x,\y) - k^2 \G^\k_{\emm}(\x,\y) =  \sum_{{\rm R} \in \Lambda}e^{i \kb \dd {\rm R}}\d_{{\rm R}}(\x - \y) \I_3,
\end{equation*}
and satisfy the boundary conditions:
\begin{equation*}
    \e_3 \t \G^\k_{{\rm e}}(\x,\y) = 0 \q \text {and} \q
    \e_3 \dd \G^\k_{{\rm m}}(\x,\y) = 0 \q \mbox{for} ~~ \x \in \Gamma, \y \in \R^3_+,
\end{equation*}
respectively. As a direct application of \eqref{expansion:green:emm}, we have the asymptotics of $\Pi^{\d \k}_{\emm}$:
\begin{equation} \label{expansion:matrix:emm}
    \Pi^{\d \k}_{\emm}(\x,\y) = \sum_{n = -1}^\infty   (\d k)^n \Pi_{n,\emm}(\x,\y).
\end{equation}
Then we readily see an expansion from the above formula and the definition of $\G^{\k}_{\emm}$ in \eqref{def:tensor:em}:
\begin{equation}\label{expan:Green:tensor}
\G^{\d \k}_{\emm}(\x,\y) = \frac{1}{\d k} \G_{-1,\emm}(\x,\y) + \sum_{n = 0}^\infty (\d k)^n \G_{n,\emm}(\x,\y)\,,
\end{equation}
where $\G_{n,\emm}(\x,\y)$ is given by
$$\G_{n,\emm}(\x,\y) = \Pi_{n,\emm}(\x,\y) +
\gg_x \gg_x \dd \Pi_{n+2,\emm}(\x,\y).$$
We end this subsection with some basic but very useful observations:
\begin{align}
    \frac{\p}{\p x_i}G^\k_\ze = - \frac{\p}{\p
    y_i}G^\k_\ze(i =1,2), \q \frac{\p}{\p x_3}G^\k_\ze = - \frac{\p}{\p y_3}G^\k_\zm ,\\
    \frac{\p}{\p x_i}G^\k_\zm = - \frac{\p}{\p y_i}G^\k_\zm(i = 1,2), \q \frac{\p}{\p x_3}G^\k_\zm = -\frac{\p}{\p y_3}G^\k_\ze,
\end{align}
which lead us to the following reciprocity:
\begin{align} \label{rela:tensor:trans}
    \gg_\x \t \Pi^\k_{\emm}(\x,\y)^T = \gg_\y \t
    \Pi^\k_{\me}(\x,\y).
\end{align}

\subsection{Integral operators and their asymptotics} \label{sec:layer potential}
With the help of the Green's tensors introduced in the last subsection, we define the following vector potentials with density $\varphi$ on $\p B$ \cite{griesmaier2008asymptotic,colton2012inverse}:
\begin{align*}
\A_{B,\emm}^\k : \htd &\longrightarrow H({\rm curl},B)\  \text{or} \ H_{loc}({\rm curl}, \Omega\backslash \bar{B})  \\
\varphi & \longmapsto \A_{B,\emm}^\k[\varphi](\x) = \int_{\p B} \Pi^\k_{\emm}(\x,\y)\varphi(\y)d\sigma; \\
\M_{B,\emm}^\k: \htd &\longrightarrow \htd  \\
\varphi & \longmapsto \M_{B,\emm}^\k[\varphi](\x) = \int_{\p B} \n(\x) \t \nabla_\x \t  \Pi_{\emm}^{\k}(\x,\y)\varphi(\y)d\sigma; \\
\L_{B,\emm}^\k: \htd &\longrightarrow \htd  \\
\varphi & \longmapsto \L_{B,\emm}^\k[\varphi](\x) =  \n(\x) \t (k^2 \A_{B,\emm}^\k[\varphi](\x) + \nabla \S_{B,\emm}^\k[\nabla_{\p B} \dd \varphi](\x)).
\end{align*}
Further, we define the single layer potential
\begin{align*}
    \S_{B,\emm}^\k: \hsn &\longrightarrow \hst  \\
\varphi & \longmapsto \S_{B,\emm}^\k[\varphi](\x) = \int_{\p B} G_{\emm}^{\k}(\x,\y)\varphi(\y)d\sigma,
\end{align*}
the double layer potential
\begin{align*}
    \K_{B,\emm}^\k: \hst &\longrightarrow \hst  \\
\varphi & \longmapsto \K_{B,\emm}^\k[\varphi](\x) = \int_{\p B} \frac{\p}{\p\n_\y}G^\k_{\emm}(\x,\y)\varphi(\y)d\sigma,
\end{align*}
and the Neumann-Poincar\'{e} operator
\begin{align*}
    \K_{B,\emm}^{\k,*}: \hsn &\longrightarrow \hsn  \\
\varphi & \longmapsto  \K_{B,\emm}^{\k,*}[\varphi](\x) = \int_{\p B} \frac{\p}{\p\n_\x}G^\k_{\emm}(\x,\y)\varphi(\y)d\sigma.
\end{align*}
It follows directly from the definition that $\S_{B,\emm}^\k$ satisfy the Dirichlet and Neumann boundary conditions,
respectively, on the reflective plane $\Gamma$, while $\A_{B,\ze}^\k$ and $\A_{B,\zm}^\k$
satisfy the following conditions, respectively:
\begin{equation*}
    \e_3 \t \A_{B,\ze}^\k[\varphi](\x) = 0,  \q \e_3 \dd \A_{B,\zm}^\k[\varphi](\x) = 0 \q \text{on}\  \Gamma\,.
\end{equation*}
When $k = 0$, we omit the subscript $B$ in all the potentials defined above,
e.g., we write $\so$ for $\S^0_{B,\emm}$.
We emphasize that all the definitions depend on the lattice $\Lambda$ and the domain in the unit cell $\Omega$. For the scaled lattice $\Lambda^\delta$ and domain $D$, all the operators above can be defined similarly. It can be shown that $\curl \A^k_{B,\emm}$ defines a bounded linear operator from $\htd$  into $H({\rm curl},B)$ or $H({\rm curl},\Omega \backslash \bar{B})$(cf.\cite{colton2012inverse}).
Noting that
$G^\k_{\emm}(\x) - G^k(\x)$ is a smooth function defined in $\Omega$,
thus
the trace formulas related to $\A^\k_{B,\emm}$ follow directly from the standard results \cite[Lemma 2.96]{ammari2018mathematical},
\begin{align}
& (\n \t \curl \A^\k_{B,\emm})|_{\pm} = \mp \frac{1}{2} + \M^\k_{B,\emm} \label{formu:label:vA},
\\ & (\n \t \curl \curl \A^\k_{B,\emm})|_{\pm} = \L^\k_{B,\emm} \label{formu:label:vB},
\end{align}
while it holds for $\S^\k_{B,\emm}$ that,
\begin{equation} \label{for:trace:S}
    (\frac{\p}{\p \n_\x} \S^\k_{B,\emm})|_{\pm} = \pm \frac{1}{2} + \K^\k_{B,\emm} .
\end{equation}
Recalling the asymptotic expansions (\ref{expansion:green:emm}) and (\ref{expansion:matrix:emm}), we may define the potentials $\A_{n,\emm},\S_{n,\emm}$ associated with $\Pi_{n,\emm}$ and $G_{n,\emm}$ respectively, and $\K_{n,\emm}$ and $\K^*_{n,\emm}$ as well. Then we can directly see that the following expansions hold for any density $\varphi$ on $\p D$,
\begin{align}
   \A_{D,\emm}^\k[\varphi](\x) = \d \A_{B,\emm}^{\d \k}[\w{\varphi}](\w{\x}) = \sum_{n = -1}^\infty \d^{n + 1} k^n \A_{n,\emm}[\w{\varphi}](\w{\x}) \label{for:asy:A}\,, \\
     \S_{D,\emm}^\k[\varphi](\x) = \d \S_{B,\emm}^{\d \k}[\w{\varphi}](\w{\x}) =
\sum_{n = -1}^\infty \d^{n + 1} k^n
\S_{n,\emm}[\w{\varphi}](\w{\x}) \label{for:asy:S}\,,
\end{align}
and
\begin{equation*}
    \K^\k_{D,\emm}[\varphi](\x) = \sum_{n = 0}^\infty \d^n k^n \K_{n,\emm}[\w{\varphi}](\w{\x}), \q   \K^{\k,*}_{D,\emm}[\varphi](\x) = \sum_{n = 0}^\infty \d^n k^n \K^*_{n,\emm}[\w{\varphi}](\w{\x}).
\end{equation*}
Moreover, by these asymptotic expansions, a similar proof to the one of \cite[Lemmas 3.1-3.2]{ammari2016plasmaxwell} yields the results in the next two lemmas.
\begin{lemma} \label{thm:asy:M}
For $\phi \in H_T^{-\frac{1}{2}}({\rm div},\p D)$, $\M_{D,\emm}^\k[\phi]$ has the following asymptotic expansion:
\begin{equation} \label{exp:asy:M}
\M_{D,\emm}^\k[\phi](\d \w{\x}) = \M_{B,\emm}^{\d \k}[\w{\phi}](\w{\x}) = \sum_{n = 0}^\infty (\d k)^n \M_{n,\emm}[\w{\phi}](\w{\x}),
\end{equation}
where $\M_{n,\emm}[\w{\phi}](\w{\x}) = \int_{\p B} \n(\w{\x}) \t \nabla_{\w{\x}} \t \Pi_{n,\emm}(\w{\x},\w{\y})\phi(\w{\y})d\sigma$, and has an uniform bound in $\L(\htd)$.
Moreover, $\M_{D,\emm}^\k$ is analytic in $\d$.
\end{lemma}
\begin{lemma} \label{thm:asy:L}
For $\phi \in H_T^{-\frac{1}{2}}({\rm div},\p D)$, $\L_{D,\emm}^\k[\phi]$ has the asymptotic expansion:
\begin{equation*}
\L_{D,\emm}^\k[\phi](\d \w{\x}) - \L_{D,\emm}^{\k_c}[\phi](\d \w{\x}) =  \sum_{n = 1}^\infty \d^{n-1} (k^n - k_c^n) \L_{n,\emm}[\w{\phi}](\w{\x}),
\end{equation*}
where
\begin{align*}
\L_{n,\emm}[\w{\phi}](\w{\x})
= \n \times \A_{n-2,\emm}[\w{\phi}](\w{\x}) + \n \times \gg \S_{n,\emm}[\divs \w{\phi}](\w{\x}).
\end{align*}
In particular, it holds that
\begin{align*}
&\L_{1,\ze}[\w{\phi}](\w{\x}) =- \frac{i}{ \tau d_3} \n(\w{\x}) \times \int_{\p B} \w{y}_3 \e_3 \divs \w{\phi}(\w{\y})d\sigma + \n(\w{\x}) \times \gg \int_{\p B} G_{1,\ze}(\w{\x},\w{\y})\divs \w{\phi}(\w{\y})d\sigma\,, \\
&\L_{1,\zm}[\w{\phi}](\w{\x}) = - \frac{i}{\tau d_3} \n(\w{\x}) \times \int_{\p B}(\w{y}',0)^t \divs \w{\phi}(\w{\y})d\sigma
 + \n(\w{\x}) \times \gg \int_{\p B} G_{1,\zm}(\w{\x},\w{\y})\divs \w{\phi}(\w{\y})d\sigma\,.
\end{align*}
Moreover, $\L_{n,\emm}$ has an uniform bound in $\L(\htd)$, and
$\L_{D,\emm}^\k$ is analytic in $\d$.
\end{lemma}

For the sake of simplicity, we write
$\M_{\emm},\K_{\emm},\K^*_{\emm}$ for the leading-order terms in the asymptotic expansions of $\M^{\d \k}_{B,\emm},\K^{\d \k}_{B,\emm},\K^{\d \k,*}_{B,\emm}$, respectively.
We emphasize that we only need the surface divergence of density $\w{\phi}$ to evaluate $\L_{1,\emm}[\w{\phi}]$, which implies immediately that
\begin{equation*}
    \csv H^{\frac{1}{2}}_0(\p B) \subset \hh \subset ker(\L_{1,\emm}),
\end{equation*}
where $\hh$ denotes the divergence free space, i.e.,
\begin{equation*}
\hh:= \{\varphi \in \htd; \ \divs \varphi = 0\}.
\end{equation*}
This observation shall be used repeatedly in Section \ref{sec:Far-field}. To have a better understanding of the terms involved in the expansions, we give the following lemma.
\begin{lemma} \label{rela:all:ml}
For any $\w{\phi} \in H_T^{-\frac{1}{2}}({\rm div},\p B)$, it holds that
\begin{enumerate}[(i)]
\item $\divs \L_{n,\emm}[\w{\phi}] = \divs (\n \times \A_{n-2,\emm})[\w{\phi}]$ for $n \ge 1$. In particular, $\divs \L_{1,\emm}[\w{\phi}] = 0$.
\item $\divs \M_{n,\emm}[\w{\phi}]  = -\K_{n,\emm}^* [\divs  \w{\phi}] - \n\dd \A_{n-2,\emm}[\w{\phi}]$ for $n \ge 1$,
while for $n = 0$,  $$\divs \M_{\emm}[\w{\phi}] = - \K^{*}_{\emm}[\divs \w{\phi}].$$
\end{enumerate}
\end{lemma}

\begin{proof}
We first note that $\gg \S_{n,\emm}[\divs \w{\phi}] \in H({\rm curl},B)$, then obtain the property $(i)$ by using \eqref{bapro:sur:1} and (\ref{relation:trace:curl}) to see that $ \divs(\n \t \gg \S_{n,\emm}[\divs \w{\phi}]) = 0$.
For the second property, we obtain for $n \ge 1$
by means of (\ref{bapro:sur:1}) and (\ref{relation:trace:curl}) that
\begin{align*}
    &\divs\M_{n,\emm}[\w{\phi}] = \divs (\n \t \pi_t(\curl \A_{n,\emm}))[\w{\phi}] \\
 = & \cs \pi_t (\curl \A_{n,\emm})[\w{\phi}] = - \gamma_n(\curl \curl \A_{n,\emm})[\w{\phi}] \\
 = & -\K_{n,\emm}^* [\divs  \w{\phi}] - \n(\w{x})\dd \A_{n-2,\emm}[\w{\phi}].
\end{align*}
We should be more careful to deal with the case $n = 0$ due to the jump of the trace, i.e.,
\eqref{formu:label:vA}, \eqref{for:trace:S}.  But a similar calculation as the one presented above gives
$\divs \M_{\emm}[\w{\phi}] = - \K^{*}_{\emm}[\divs \w{\phi}]$.
\end{proof}

Before we move on to the next subsection on spectral analysis, we make more investigation into the leading-order terms to prepare some tools for the later use. Recalling formulas (\ref{rela:leading-ordere})-(\ref{rela:leading-orderm}),  we know that
$\K^*_{\ze} = \K_\ze^{0,*}$ with the adjoint operator $\K_\ze =\K_\ze^0$. However, $\K_\zm^*$ can only be identified
with $\K_\zm^{0,*}$ on $H_0^{-\frac{1}{2}}(\p B)$, and $\K_\zm$ here
is not the adjoint operator of $\K^*_\zm$.  Indeed, the adjoint operators of $\K^*_\zm$ and $\K_\zm$ are defined by
\begin{equation}
  \hat{\K}_\zm[\varphi](\x) = \int_{\p B}\frac{\p}{\p \n_\y} \hat{G}_{0,\zm}(\x,\y)\varphi(\y)d\sigma
  \ \text{~~and~~}\ \hat{\K}^*_\zm[\varphi](\x) = \int_{\p B}\frac{\p}{\p \n_\x} \hat{G}_{0,\zm}(\x,\y)\varphi(\y) d\sigma
\end{equation}
for smooth function $\varphi$, respectively; see \eqref{kernel:conjugate:e}-\eqref{kernel:conjugate:m} for the definition of $\h{G}_{0,\emm}$. To find the adjoint operator of $\M_{\emm}$,
we now introduce the conjugate matrix-valued function $\h{\Pi}_{\emm}$ of $\Pi_{\emm}$:
$$ \h{\Pi}_{\emm}(\x,\y) =
     \mm
     \h{G}_{0,\emm}\e_1,& \h{G}_{0,\emm}\e_2,& \h{G}_{0,\me}\e_3
     \nn(\x,\y), $$
and the associated layer potential $\h{\M}_{\emm}:$
\begin{equation*}
    \h{\M}_{\emm}[\varphi](\x) = \int_{\p B} \n(\x) \t \nabla_\x \t  \h{\Pi}_{\emm}(\x,\y)\varphi(\y)d\sigma,
\end{equation*}
which is a bounded linear operator from $\htd$ to $\htd$.
Then, the adjoint operator of $\M_{\emm}$, i.e., $\M^*_{\emm}: \htc \to \htc$, is given by
\begin{equation} \label{for:dualofm}
    \M_{\emm}^{*} = r \hat{\M}_{\me} r\,.
\end{equation}
Actually, by a standard density argument, it suffices to verify this for smooth functions.
Using  (\ref{rela:tensor:trans}) and Fubini's theorem, we can write for smooth functions $\psi, \phi$:
\begin{align*}
\l\psi, \M_{\emm}[\phi]\r  = & \int_{\p B} \int_{\p B}  \psi(\x) \dd \n(\x) \t \gg_\x \t (\Pi_{\emm}(\x,\y)\phi(\y))d\sigma(\y)d\sigma(\x)\\
= & \int_{\p B} \int_{\p B} (\gg_\x \t \Pi_{\emm}(\x,\y))^{\ts} (\psi(\x) \t \n(\x)) \dd \phi(\y) d\sigma(\y)d\sigma(\x) \\
= & \int_{\p B} \int_{\p B} (\gg_\y \t \h{\Pi}_{\me}(\y,\x)) (\psi(\x) \t \n(\x)) \dd \phi(\y) d\sigma(\x)d\sigma(\y) \\
= & \l r \hat{\M}_{\me} r[\psi] ,\phi\r.
\end{align*}
Similarly, we can get the adjoint operator $\hat{\M}^*_{\emm}: \htc \to \htc$ of $\hat{\M}_{\emm}$:
\begin{equation} \label{aux:hm:adjoint}
    \hat{\M}^*_{\emm} = r\M_{\me}r.
\end{equation}
Recall that we have proven in Lemma \ref{rela:all:ml}
\begin{equation} \label{Mdivergence}
    \divs\M_{\emm}[\varphi] = - \K^{*}_{\emm}[\divs \varphi] =  - \K^{0,*}_{\emm}[\divs \varphi].
\end{equation}
The last equality above is due to the fact that $\divs \varphi \in H_0^{-\frac{1}{2}}(\p B)$, on which $\K_{\emm}$ and $\K^{0,*}_{\emm}$ can be identical. By exactly the same arguments, we obtain
\begin{equation} \label{aux:hM:div}
    \divs \h{\M}_{\emm}[\varphi] = - \h{\K}^{*}_{\emm}[\divs \varphi]= - \K^{0,*}_{\emm}[\divs \varphi].
\end{equation}
Taking the adjoint on the both sides of (\ref{aux:hM:div}) and using (\ref{aux:hm:adjoint}), we can see that
\begin{equation} \label{Mcurl}
    \M_{\me}\csv = \csv \K^0_{\emm}.
\end{equation}

\subsection{Spectral analysis of integral operators}
In this subsection we are going to consider the spectral properties of  Neumann-Poincar\'{e} type operators, which is essential for the subsequent analysis of the blow-up order of the scattered field. We start with some basic facts, and the interested reader are referred to \cite{ammari2018mathematical,ammari2009layer,ammari2016mathematical} for more details. Considering the single layer potential $\S^0_{\emm}$, it is easy to observe that $\S_{\emm}^0 : \hsn \to \hst $ is self-adjoint, i.e., $\l\psi, \S_{\emm}^0[\phi]\r = \l\S_{\emm}^0[\psi],\phi\r$, and the Calder\'{o}n identity:
\begin{equation} \label{eq:calid}
    \S_{\emm}^0 \K_{\emm}^{0,*} = \K_{\emm}^0 \S_{\emm}^0
\end{equation}
holds in $\hsn$. However, since $\S_{\zm}^0$ is generally not invertible nor injective on $\hsn$, the standard symmetrization technique via Calder\'{o}n identity \eqref{eq:calid} is no longer applicable. Indeed, $\S_{\ze}^0$ is injective on $H^{-\frac{1}{2}}(\p B)$ while $\S_{\zm}^0$ is injective only on $H_0^{-\frac{1}{2}}(\p B)$. Moreover, the dimension of the kernel of $\S_{\zm}^0$ in $H^{-\frac{1}{2}}(\p B)$ is at most $1$ (under the assumption that $\p B$ is connected). To see this, we first observe the far-field behavior of $\S^0_{\emm}[\phi]$ from (\ref{quasi-green:k=0}), i.e.,
it holds for $\phi \in \hsn$ and large enough $x_3$ that
\begin{align} \label{far-field:single}
    \S_{\emm}^0[\phi](\x) = c_{0,\emm}(\phi) + \sum_{\xi \in \Lambda^*\backslash \{0\}}\frac{1}{|\xi|}c_{\xi,\emm}(\phi) e^{i\xi \dd x'}e^{-|\xi|x_3},
\end{align}
where the coefficients $c_{0,\ze}(\phi)$ and $c_{0,\zm}(\phi)$ are given by
\begin{align} \label{far-field:single2}
    c_{0,\ze}(\phi) = -\frac{1}{\tau}\int_{\p B} y_3 \phi(\y) d\sigma({\y}), \q c_{0,\zm}(\phi) = \frac{x_3}{\tau}\l \phi , 1 \r.
\end{align}
Then we have by using integration by parts,
\begin{equation*}
    \int_{\sss \Sigma\t(0,L)} |\nabla \S^0_{\emm}[\phi]|^2 d\x = - \int_{\p B}\phi \overline{\S^0_{\emm}}[\phi] d\sigma + \int_{\sss \Sigma \t \{L\}} \frac{\p \S^0_{\emm}[\phi]}{\p \n}\overline{\S^0_{\emm}}[\phi] d\sigma,
\end{equation*}
which, combined with \eqref{far-field:single} and \eqref{far-field:single2}, implies that, by letting $L$ tends to infinity,
\begin{equation} \label{eq:posisinge}
    \int_{\Omega} |\nabla \S^0_{\ze}[\phi]|^2 d\x = - \int_{\p B} \phi  \overline{\S^0_{\ze}}[\phi] d\sigma \ge 0
\end{equation}
holds for all $\phi \in \hsn$, while \eqref{eq:posisinge} holds only for $\phi \in H_0^{-\frac{1}{2}}(\p B)$
when $\S^0_{\ze}$ is replaced by $\S^0_{\zm}$.

There is a standard way to overcome this difficulty (cf. \cite[Theorem 2.26]{ammari2007polarization},  \cite{kang2016spectral,ando2016analysis}), whose main idea is given below for convenience.
Introduce the bounded operator $A_{\emm}:\hsn \t \C \rightarrow \hst \t \C$ by
\begin{equation*}
    A_{\emm}(\phi,a) := (\S^0_{\emm}[\phi]+a, \l \phi, 1\r),
\end{equation*}
which can be shown to have a bounded inverse. In fact,
since the Fredholm index is unchanged under compact perturbation, we can conclude that $A_{\emm}$ is Fredholm with zero index. Hence it suffices to prove the injectivity to establish the invertibility, which follows exactly from the same proof as in \cite[Theorem 2.26]{ammari2007polarization}. Then we can prove that $\so$ is invertible if and only if $\so[\varphi^{\emm}_0] \neq 0$ (cf.\cite{ando2016analysis}), where $\varphi^{\emm}_0$ is the eigenfunction of $\no$ associated with the eigenvalue $\frac{1}{2}$, satisfying $\l \varphi^{\emm}_0 , 1\r = -1$.  We now define
\begin{equation*}
\w{\S}_{\emm}^0[\psi] = \begin{cases}
\so[\psi]   &\mbox {if ~ $\l \psi ,1\r = 0 $}, \\
1   &\mbox{if ~ $\psi = \varphi^{\emm}_0$}.
\end{cases}
\end{equation*}
Then $\w{\S}_{\emm}^0$ is a bijection from $\hsn$ to $\hst$, and the generalized Calder\'{o}n identity holds:
$$\w{\S}_{\emm}^0 \no = \ko \w{\S}_{\emm}^0.$$
This allows us to define
two new inner products on $\hsn$, equivalent to the original one, such that $\no$ is self-adjoint,
$$(\phi, \psi)_{\H_{\emm}^*} =  - \l \phi ,\w{\S}^0_{\emm}[\psi] \r\,.$$
We denote by $\Hs_{\emm}$ the space $\hsn$ equipped with these two new inner products, respectively. Then we can symmetrize $\no$ as it is stated below.
\begin{lemma} \label{thm:spec:no}
For a $C^{2}$ bounded domain $B$ with a connected boundary, we have
\begin{enumerate}[(i)]
\item $\no$ is compact and self-adjoint on the Hilbert space $\H_{\emm}^*$.
\item Suppose that $(\lambda^{\emm}_j,\varphi^{\emm}_j)$ is the eigenvalue and normalized eigenfunction pair of $\no$ with $\lambda^{\emm}_0 = \frac{1}{2}$, then $\lambda^{\emm}_j\in(-\frac{1}{2},\frac{1}{2}]$ with $\lambda^{\emm}_j \to 0$ as $j \to \infty$.
\item $\{\varphi^{\emm}_j\}$ is an orthogonal basis in $\H_{\emm}^*$. More precisely, $\H_{\emm}^* = \H^*_{0,\emm} \oplus \{\mu \varphi^{\emm}_0, \mu \in \C \}$, where $\H_{0,\emm}^*$ is the zero mean subspace of $\H_{\emm}^*$ spanned by $\{\varphi^{\emm}_j\}_{j \neq 0}$.
\item The following spectral decomposition holds,
\begin{equation} \label{eq:specdecom}
  \no[\phi] =  \sum_{j = 0}^\infty \lambda^{\emm}_j (\phi, \varphi^{\emm}_j)_{\Hs_{\emm}}\varphi^{\emm}_j.
\end{equation}
\end{enumerate}
\end{lemma}
Similarly,
we can define the inner products on $H^{\frac{1}{2}}(\p B)$ by
 $$-\l(\w{\S}_{\emm}^0)^{-1}[\psi], \phi\r,$$
 and denote  by $\H_{\emm}$ the Hilbert space $\hst$ equipped with these two inner products respectively,
then the norm equivalence holds,
i.e., $\norm{u}_{\H_{\emm}} \approx \norm{u}_{\sss H^{\frac{1}{2}}(\p B)}$. Note that $\w{\S}_{\emm}^0$ is an unitary operator from $\H^*_{\emm}$ to $\H_{\emm}$, hence $\{\w{\S}_{\emm}^0[\varphi_j]\}$ is an orthogonal basis on $\H_{\emm}$. We point out that $\w{\S}_{\emm}^0[\varphi_j]$ is actually the eigenfunctions of $\ko$. Now we are ready to consider the leading-order terms $\K^*_{\emm}$ and $\K_{\emm}$ in the expansions of
$\K^{\d \k,*}_{B,\emm}$ and $\K^{\d \k}_{B,\emm}$,
and they can be regarded as the corrections of $\K^{0,*}_{\emm}$ and $\K^0_{\emm}$
due to the incident angle. In fact, recalling the definitions of $\K_{\emm}$ and $\K^*_{\emm}$, and using \eqref{rela:leading-ordere}-\eqref{rela:leading-orderm}, we obtain
\begin{align}
   & \K_{\ze} = \K^0_{\ze}, \q  \K_{\zm}[\phi] = \K^0_{\zm} + \frac{1}{d_3 \tau}\l d' \dd \n' \phi, 1\r, \label{for:correctinp}\\
   & \K^*_{\ze} = \K^{0,*}_{\ze}, \q   \K^*_{\zm}[\phi] = \K^{0,*}_{\zm} - \frac{d'\dd \n'}{d_3 \tau}\l\phi,1\r. \label{for:correctidou}
\end{align}
Hence the spectral structure of $\K^*_{\ze}$ can be completely characterized by Lemma \ref{thm:spec:no}.
We shall only pay attention to $\K^*_{\zm}$ below, and it turns out that its spectra has nothing to do with the incident angle although there are remaining items in \eqref{for:correctinp} and \eqref{for:correctidou} that are related. We now present several spectral results
for that we introduce some standard notation. For a compact operator $K$, we denote by $\sigma(K)$ its spectrum set and by $(\lambda I - K)^{-1}$ its resolvent operator for regular points $\lambda \in \C \backslash \sigma(K)$.
For point $p$ and set $F$ in complex plane $\C$, we define their distance
$ 
    d(p,F):=\inf_{q \in F}|p-q|\,.
$ 
\begin{theorem} \label{thm:specnp}
The operators $\K_\zm^*$ and $\K_\zm^{0,*}$ have the same spectra. Further, for $\lambda_j \in \sigma(\K^*_\zm)\backslash\{0\}$, we have
 ${\rm dim} ker(\lambda_j - \K_\zm^*) = {\rm dim} ker(\lambda_j - \K_\zm^{0,*})$.
\end{theorem}
\begin{proof}
It is known that $\K_\zm^*$ is a compact operator with adjoint operator $\h{\K}_{\zm}$ and $\h{\K}_\zm[1] = \frac{1}{2}$ holds by definition, which implies that $\frac{1}{2}$ is also an eigenvalue of $\K_\zm^*$. Combining this with the fact that $\K_\zm^* = \K_\zm^{0,*}$ on $H^{-\frac{1}{2}}_0(\p B)$, we have
$\sigma(\K^{0,*}_\zm) \subset \sigma(\K_\zm^*)$. Suppose $\lambda \in \sigma(\K_\zm^*)\backslash\{0,\frac{1}{2}\}$ and that $\phi$ is the associated eigenfunction, then we obtain by using $\h{\K}_\zm[1] = \frac{1}{2}$ that
\begin{equation*}
 0 = \int_{\p B}(\lambda I - \K_\zm^*)[\phi]d\sigma = (\lambda - \frac{1}{2})\int_{\p B}\phi d\sigma\,,
\end{equation*}
which further yields
$\phi \in H^{-\frac{1}{2}}_0(\p B)$. We thus have $\sigma(\K^{0,*}_\zm) = \sigma(\K_\zm^*)$, and the desired result
follows.
\end{proof}
We next consider the spectral decomposition of $\K^{*}_{\zm}$ and $\K_{\zm}$ and the corresponding resolvent estimates.
Suppose that $\phi \in \Hs_{\zm}$ has the following decomposition with respect to the orthogonal basis $\{\varphi_j\}$ given by the eigensystem $\{\lambda_j,\vp_j\}$ of $\K_\zm^{0,*}$:
\begin{equation} \label{decom:ortho}
\phi = \sum_{j = 0}^\infty (\phi, \vp_j)_{\Hs_{\zm}}\vp_j.
\end{equation}
Here we have omitted the subscript or superscript $\zm$ for simplicity.
By writing its Fourier coefficients $(\phi,\vp_j)_{\Hs_{\zm}}$ as $\h{\phi}(j)$ and using Lemma \ref{thm:spec:no} and Theorem \ref{thm:specnp},
we can derive
\begin{align}
\K_\zm^*[\phi]& = \h{\phi}(0) \K_\zm^*[\varphi_0] + \sum_{n = 1}^\infty \h{\phi}(j)  \lambda_j \vp_j
 = \h{\phi}(0) (\frac{1}{2}\vp_0 + \sum_{j =1}^\infty \iota_j \vp_j) + \sum_{j = 1}^\infty \h{\phi}(j)  \lambda_j \vp_j \nb\\
& = \sum_{j = 0}^\infty  \lambda_j\h{\phi}(j) \vp_j  + \h{\phi}(0)  \sum^\infty_{j = 1} \iota_j \vp_j, \label{eq:Km*}
\end{align}
where the constants $\{\iota_j\}^\infty_{j=1}$ are obtained by applying \eqref{decom:ortho} to $-\frac{d'\dd \n'}{d_3 \tau}\l\vp_0,1\r$, i.e., $$ -\frac{d'\dd \n'}{d_3 \tau}\l\vp_0,1\r = \frac{d'\dd \n'}{d_3 \tau} = \sum_{j =1}^\infty \iota_j \vp_j.$$
\begin{proposition} \label{thm:reso:np}
For $f \in \hsn$, the following resolvent estimate holds
\begin{align*}
\norm{(\lambda I - \np)^{-1}[f]}_{\H_{\emm}^*} \lesssim \frac{\norm{f}_{\Hs_{\emm}}}{d(\lambda,\sigma(\K^{0,*}_{\emm}))}.
\end{align*}
\end{proposition}

\begin{proof}
The resolvent estimate of $\K_\ze^*$ follows directly from the fact that $\K_\ze^*$ is a compact self-adjoint operator on  $\Hs_\ze$. For $\K_\zm^*$,
considering the equation $(\lambda I -  \K_\zm^*)[\phi] = f$, we obtain from \eqref{eq:Km*} that
\begin{align*}
\sum_{j = 0}^\infty  (\lambda - \lambda_j)\h{\phi}(j)\vp_j  - \h{\phi}(0) \sum^\infty_{j = 1} \iota_j\vp_j  = \sum_{j = 0}^\infty \h{f}(j)  \vp_j.
\end{align*}
For $\lambda \notin \sigma(\K_\zm^{0,*})$, $\h{\phi}(j)$ can be uniquely determined by
\begin{align*}
\h{\phi}(0) = \frac{1}{\lambda - \frac{1}{2}}\h{f}(0)\,; \quad
\h{\phi}(j) = \frac{\h{f}(j) +\h{\phi}(0)\iota_j}{\lambda - \lambda_j}
= \frac{\h{f}(j)}{\lambda - \lambda_j} + \frac{\h{f}(0)\iota_j}{(\lambda-\lambda_j)(\lambda-\frac{1}{2})} \q \mbox{for} ~~ j \ge 1.
\end{align*}
Using the above formulas, we then derive
\begin{align*}
\norm{\phi}_{\H_\zm^*} &\lesssim \frac{\norm{f}_{\Hs_\zm}}{d(\lambda,\sigma(\K_\zm^{0,*}))} + \frac{|\l f, 1\r|}{d(\lambda,\sigma(\K_\zm^{0,*})\backslash\{\frac{1}{2}\})\dd |\lambda - \frac{1}{2}|}
\lesssim \frac{\norm{f}_{\Hs_\zm}}{d(\lambda,\sigma(\K_\zm^{0,*}))}.
\end{align*}
\end{proof}
For convenience, we shall define $\c{\psi}(j) : = (\psi,\w{\S}^0_{\zm}[\varphi_j])_{\H_{\zm}}$ for $\psi \in \H_{\zm}$.
Then we can write
\begin{equation}
    \psi = \sum_{j = 0}^\infty \c{\psi}(j)\w{\S}^0_{\zm}[\varphi_j]. \label{eq:Sem}
\end{equation}
Using this and similar arguments to the ones in the proof of Proposition\,\ref{thm:reso:np}, we can obtain the resolvent estimate of $\K_{\emm}$.
\begin{proposition} \label{thm:reso:kk}
For any $g \in \hst$, we have the resolvent estimate:
\begin{align*}
\norm{(\lambda I - \kk)^{-1}[g]}_{\H_{\emm}} \lesssim \frac{\norm{g}_{\H_{\emm}}}{dist(\lambda,\sigma(\K^0_{\emm}))}.
\end{align*}
\end{proposition}

\begin{proof}
Again, we prove the estimate only for $\K_\zm$. It is easy to see that
\begin{align*}
\K_\zm[\psi] = &\sum\limits_{j = 0}^\infty \c{\psi}(j) \lambda_j\w{\S}^0_\zm[\vp_j] + \frac{1}{d_3 \tau} \sum\limits_{j = 0}^\infty(- d'\dd \n',\vp_j)_{\Hs_\zm} \c{\psi}(j) \\
= & \sum_{j = 0}^ \infty\c{\psi}(j)\lambda_j \w{\S}^0_\zm[\vp_j] - \sum_{j = 1}^\infty \iota_j \c{\psi}(j).
\end{align*}
Considering the equation $(\lambda I - \K_\zm)[\psi] = g$, and using \eqref{eq:Sem} we write
$$ \sum^\infty_{j=1}(\lambda - \lambda_j)\c{\psi}(j) \w{\S}^0_{\zm}[\vp_j] + \sm{1} \iota_n \c{\psi}(n) = \c{g}(0) + \sm{1} \c{g}(n)\w{\S}^0_\zm[\varphi_n]. $$
For $\lambda \notin \sigma(\K_\zm^0)$, $\c{\psi}(n)$ can be uniquely determined by
\begin{align*}
\c{\psi}(n)  = \frac{\c{g}(n)}{\lambda - \lambda_n}, \q n \ge 1\,; \quad
\c{\psi}(0)  = \frac{1}{\lambda - \frac{1}{2}} (\c{g}(0) - \sm{1} \iota_n \c{\psi}(n))
= \frac{1}{\lambda - \frac{1}{2}} (\c{g}(0) - \sm{1}\iota_n \frac{\c{g}(n)}{\lambda - \lambda_n}).
\end{align*}
%
Then we can obtain the desired estimate:
\begin{align*}
\norm{\psi}_{\H_\zm} &\lesssim \frac{\norm{g}_{\H_\zm}}{d(\lambda,\sigma(\K^0_\zm))} + \frac{\norm{g}_{\H_\zm}}{d(\lambda,\sigma(\K^0_\zm)\backslash\{\frac{1}{2}\}) \dd |\lambda - \frac{1}{2}|}
\lesssim \frac{\norm{g}_{\H_\zm}}{d(\lambda,\sigma(\K^0_\zm))}.
\end{align*}
\end{proof}

The spectral results in this subsection suggest us that in most cases, there is no need to distinguish between
$\K_{\emm}$ and $\K^0_{\emm}$, as well as between $\K^*_{\emm}$ and $\K^{0,*}_{\emm}$, since they have the same spectrum and enjoy the same resolvent estimate.
Now, we are in a position to study the spectral structure of $\M_{\emm}$.
For each $u \in \htd$, we may recall the Helmholtz decomposition to write
\begin{equation}\label{eq:two_components}
u = \gs u^{\sss (1)} + \csv u^{\sss (2)}
\end{equation}
for two functions $u^{\sss (1)}\in H^{\frac{3}{2}}_0(\p B)$ and  $u^{ \sss (2)}\in H_0^{\frac{1}{2}}(\p B)$.
This notation will be adopted from now on, and the two subspaces corresponding to $u^{\sss (1)}$ and $u^{ \sss (2)}$
may not be always specified.
By applying the invertibility of the Laplace-Beltrami operator $\lap: H_0^{\frac{3}{2}}(\p B) \to H_0^{-\frac{1}{2}}(\p B)$ and
the inverse mapping theorem, we know the existence of an isomorphism between
$\htd$ and $H_0^{\frac{3}{2}} \t H_0^{\frac{1}{2}}$, which results in an equivalent norm on $\htd$:
\begin{align*}
\norm{\phi}_{\htd} \approx \norm{\lap\phi^{\sss (1)}}_{\hsn} + \norm{\phi^{\sss  (2)}}_{\hst}.
\end{align*}

\begin{theorem} \label{thm: spec:m}
The spectra $\sigma(\M_{\emm})$ and $ \sigma(\M_{\emm}^0)$ of the operators $\M^0_{\emm}$ and $\M_{\emm}$ are given by
\begin{equation}\label{spec:Mem}
    \sigma(\M_{\emm}) = \sigma(\M_{\emm}^0) = (- \sigma(\no)\bigcup\sigma(\K_{\me}^{0,*}))\backslash\{-\frac{1}{2},\frac{1}{2}\}.
\end{equation}
\end{theorem}

\begin{proof}
We show only the spectral property of $\M_{\emm}$, as the analysis for $\M^0_{\emm}$ is similar and even simpler. Denote by $F_{\emm}$ the set in the right-hand side of \eqref{spec:Mem}.
Define
\begin{align*}
    &\sigma^1_{\emm} := F_{\emm} \cap \sigma(\K^{0,*}_{\me}), \quad
    \sigma^2_{\emm} := F_{\emm} \backslash  \sigma(\K^{0,*}_{\me}).
\end{align*}
Since $\M_{\emm}$ is a compact operator, it suffices to consider the equation
for a given $\lambda \in \C\backslash \{0\}$,
\begin{equation} \label{eq:specM}
    (\lambda I - \M_{\emm})[\phi] = 0,
\end{equation}
and prove that it has nontrivial solutions if and only if $\lambda \in \sigma^1_{\emm} \cup \sigma^2_{\emm}$.
Using  \eqref{eq:two_components}, we can write
\begin{equation*}
 \phi = \nabla_{\p B} \phi^{\sss (1)} + \csv \phi^{\sss (2)}.
\end{equation*}
For nonzero $\lambda \in \sigma^1_{\emm}$, we first note from \eqref{Mcurl} that
\begin{equation} \label{eq:sigma_1}
    (\lambda I - \M_{\emm})[\csv \phi^{\sss (2)}] = \lambda \csv \phi^{\sss (2)} - \csv\K^0_{\me}[\phi^{\sss (2)}],
\end{equation}
which directly implies that $(\lambda, \csv\phi^{\sss (2)})$ is an eigenpair of $\M_{\emm}$ if $\phi^{\sss (2)}$ is an eigenfunction of $\K^{0}_{\me}$ associated with $\lambda$. If $\lambda \in \sigma^2_{\emm}$,
 we readily obtain by using the surface divergence for \eqref{eq:specM} that
\begin{align} \label{proo:div}
\divs (\lambda I - \M_{\emm})[\phi] = (\lambda I + \no)[\divs \phi] = (\lambda I + \no)[\lap \phi^{\sss (1)}] = 0.
\end{align}
Since the eigenfunction of $ - \K^{0,*}_{\emm}$ associated with $\lambda \in \sigma^2_{\emm}$ has mean value zero and $\lap$ is an isomorphism from $H_0^{\frac{3}{2}}(\p B)$ to $H_0^{-\frac{1}{2}}(\p B)$, we know that there exists a non-constant function $\phi^{\sss (1)}$ satisfying equation \eqref{proo:div}. We then reduce \eqref{eq:specM} via \eqref{eq:sigma_1} to
\begin{equation*}
    \lambda \csv \phi^{\sss (2)} - \csv\K^0_{\me}[\phi^{\sss (2)}] = - (\lambda I - \M_{\emm})[\gs \phi^{\sss (1)}].
\end{equation*}
Taking the surface curl on both sides of the above equation, we can find it is solvable by the invertibility of $\Delta_{\p B}$ and $\lambda I - \K^0_{\me}$. Hence, there exists a nontrivial $\phi$ satisfying equation \eqref{eq:specM} for $\lambda \in \sigma^2_{\emm}$. We are now in a position to consider the last case: $\lambda \in \C\backslash (\sigma^1_{\emm} \cup \sigma^2_{\emm})$. It is easy to derive that $\phi$ must be $\csv \phi^{\sss (2)}$ for some $\phi^{\sss (2)}$ by the invertibility of $\lambda I + \K^{0,*}_{\emm}$ on  $H_0^{-\frac{1}{2}}(\p B)$. Then the reduced equation from \eqref{eq:sigma_1} reads as follows:
$$(\lambda I - \K^0_{\me})[\phi^{\sss (2)}] = C,$$
by using the invertibility of $\lap$, where $C$ is some constant. Without loss of generality, we assume that $C =1$ or $0$. If $\lambda = \frac{1}{2}$, we must have $C = 0$ in order to guarantee the existence of $\phi^{\sss (2)}$ due to the Fredholm alternative.  In this case, we have $\phi = \csv \phi^{\sss (2)} = 0$. If $\lambda \neq \frac{1}{2}$, we can find a constant $C'$ such that
\begin{equation*}
(\lambda I - \K_{\me}^0)[\phi^{\sss (2)} + C'] = 0,
\end{equation*}
which yields $\phi^{\sss (2)}$ is a constant. Hence, if $\lambda \in \C\backslash (\sigma^1_{\emm} \cup \sigma^2_{\emm})$,
we can conclude $\phi = 0$, hence completes the proof.
\end{proof}

\section{Approximation of the scattered wave}\label{sec:Far-field}
\subsection{Integral formulation and asymptotic analysis}
With the analytical tools and results established in the previous section, we shall first reformulate the system (\ref{model}) into an boundary integral equation, then build up a norm estimate of the
associated solution operator, from which we can predict the occurrence of the resonance phenomenon.
Taking advantage of the vector potential $\A^\k_{D,\emm}$ given in the section {\ref{sec:layer potential}}, we assume
the following ansatz for the electric field solution of (\ref{model}):
\begin{equation*}
E = \begin{cases}
\ei + \curl \A_{D,\zm}^\k[\phi] + \curl \curl \A_{D,\ze}^\k[\psi]\,, & \mbox {$\x \in \R^3 \backslash \D$}
\\
\muc \curl \A_{D,\zm}^{\k_c}[\phi] + \curl \curl \A_{D,\ze}^{\k_c}[\psi]\,, & \mbox {$\x \in \D$}.
\end{cases}
\end{equation*}
It can be checked directly that the field $E$ given above
solves the Maxwell equations in both $\D$ and $\R_+^3\backslash \D$,
and satisfies the perfect conducting boundary condition on $\Gamma$. Then by the jump formula (\ref{formu:label:vA}), the original scattering problem can be equivalently written as a boundary integral equation on $\p D$:
\begin{equation*}
\mm
\frac{\muc + 1}{2} + \muc \M_{D,\zm}^{\k_c}- \M_{D,\zm}^{\k} & \L_{D,\ze}^{\k_c}-\L_{D,\ze}^\k \\
\L_{D,\zm}^{\k_c}-\L_{D,\zm}^\k &
\frac{k^2}{2}(\eec + 1)I + k^2 (\eec \M_{D,\ze}^{\k_c} - \M_{D,\ze}^\k)
\nn
\mm
\phi
\\
\psi
\nn =
\mm
\nt \ei\\
i k \nt \hi
\nn.
\end{equation*}
By setting $\x = \d \w{\x}$, we obtain an integral equation defined on $\p B$:
\begin{equation}
\W_{\d,B}
\mm
\w{\phi}\\
\w{\psi}
\nn
=
\mm
\n(\w{\x})\times \w{E}^i \\
ik \n(\w{\x}) \times \w{H}^i
\nn, \label{eq:WdeltaB}
\end{equation}
where the block coefficient matrix is given by
\begin{equation*}
\W_{\d,B} =
\mm
\frac{\muc + 1}{2} + \muc \M_{B,\zm}^{\d \k_c} - \M_{B,\zm}^{\d \k} &  \L_{\ze,\d}^{\k_c} - \L_{\ze,\d}^\k \\
\L_{\zm,\d}^{\k_c} - \L_{\zm,\d}^\k  & k^2(\frac{\eec+ 1}{2}I + \eec \M_{B,\ze}^{\d \k_c}-\M_{B,\ze}^{\d  \k})
\nn.
\end{equation*}
By
Lemma \ref{thm:asy:M} and \ref{thm:asy:L}, we have the asymptotic expansion of $\W_{\d,B}$:
\begin{equation*}
\W_{\d,B} = \sum_{n =0}^\infty \d^n \W_{n,B},
\end{equation*}
where
\begin{equation*}
\W_{0,B} =
\mm
\frac{\muc + 1}{2} + (\muc - 1)\M_\zm  & (k_c - k)\L_{1,\ze} \\
(k_c - k)\L_{1,\zm}  &  k^2\frac{\eec+ 1}{2}I + k^2(\eec - 1)\M_{\ze}
\nn,
\end{equation*}
and for $n \ge 1$,
\begin{equation*}
\W_{n,B} =
\mm
(k_c^n - k^n)\M_{n,\zm} & (k_c^{n + 1} - k^{n+1})\L_{n+1,\ze} \\
 (k_c^{n+1} - k^{n+1})\L_{n+1,\zm} &  k^2(\varepsilon_c k_c^n-k^n)\M_{n,\ze}
\nn \,.
\end{equation*}
Before we turn to finding the approximate scattered wave, we investigate the property of the leading-order term $\W_{0,B}$ first.
From now on, we introduce two contrast parameters $\lm(\omega)$ and $\lee(\omega)$:
$$\lm(\omega) = \frac{1 + \muc(\omega)}{2(1 - \muc(\omega))}, \q  \lee(\omega) = \frac{1 + \eec(\omega)}{2(1 - \eec(\omega))}.$$
\begin{theorem} \label{thm: perp : W0B}
Suppose $\lambda_\mu(\omega), \lambda_\varepsilon(\omega) \notin \sigma(\M_{\emm})$, then $\W_{0,B}$ is invertible with the estimate:
\begin{equation} \label{esti:norm:wob}
     \norm{\W^{-1}_{0,B}} \lesssim \frac{1}{d_\sigma d^*_\sigma},
\end{equation}
where the two constants $d_\sigma$ and $d^*_\sigma$ are defined by
\begin{align*}
& d_\sigma = min\{d(\lm,\sigma(\K_\ze^0)),d(\lee,\sigma(\K_\zm^0))\}, \q
d^*_\sigma = min\{d(\lm,-\sigma(\K_\zm^{0,*})),d(\lee,-\sigma(\K_\ze^{0,*}))\}.
\end{align*}
\end{theorem}
\begin{proof}
Without loss of generality, we assume that $\muc \neq 1$, $\eec \neq 1$ and consider the system
$$\W_{0,B}\mm \phi \\ \psi \nn = \mm (1-\muc)f \\ (1-\eec) g\nn $$
for given $f,g \in \htd$, which is equivalent to the following two equations:
\begin{equation}\label{in:theoprf:sys}
(\lambda_\mu I - \M_\zm)[\phi] + \frac{k_c- k}{1 - \muc} \L_{1,\ze}[\psi] = f\,, \q
\frac{k_c - k}{k^2(1 - \varepsilon_c)}\L_{1,\zm}[\phi] + (\lambda_\varepsilon I - \M_{\ze})[\psi] = g.
\end{equation}
We shall reduce (\ref{in:theoprf:sys}) to some easily solved subproblems by using the Helmholtz decomposition.
To do so, we take the surface divergence on both sides of two equations in \eqref{in:theoprf:sys},
then use formula (\ref{Mdivergence}) to obtain
\begin{align*}
(\lambda_\mu + \K_\zm^{0,*})[\divs \phi] = \divs f,  \q (\lambda_\varepsilon + \K_\ze^{0,*})[\divs \psi] = \divs g\,,
\end{align*}
which, along with the fact that $\divs u = \lap u^{\sss (1)}$ for any $u \in \htd$, yields
\begin{align}
\phi^{\sss (1)} = \lap^{-1}(\lambda_\mu + \K_\zm^{0,*})^{-1}(\lap f^{\sss (1)}), \q
\psi^{\sss (1)} = \lap^{-1}(\lee + \K_\ze^{0,*})^{-1}(\lap g^{\sss (1)}).
\end{align}
Then it follows directly from Proposition \ref{thm:reso:np} that
\begin{align*}
\norm{\lap \phi^{\sss (1)}}_{\Hs_\zm} \lesssim \frac{\norm{\lap f^{\sss (1)}}_{\Hs_\zm}}{d(\lm, - \sigma(\K^{0,*}_\zm))}, \q
\norm{\lap \psi^{\sss (1)}}_{\Hs_\ze} \lesssim \frac{\norm{\lap g^{\sss (1)}}_{\Hs_\ze}}{d(\lee, - \sigma(\K^{0,*}_\ze))} .
\end{align*}

Next, we solve the second component $\phi^{\sss (2)}$. To this purpose, we use (\ref{Mcurl}) and write the first equation in (\ref{in:theoprf:sys}) as
\begin{align*}
& (\lambda_\mu I - \M_\zm)[\csv\phi^{\sss (2)}] = \csv (\lm I - \K^0_e)[\phi^{\sss (2)}]
=  f - \frac{k_c- k}{1 - \muc} \L_{1,\ze}[\psi] - (\lm I - \M_\zm)[\gs \phi^{\sss (1)}].
\end{align*}
Taking the surface scalar curl on both sides of the equation and then applying Proposition \ref{thm:reso:kk} gives
\begin{align*}
\norm{\phi^{\sss (2)}}_{\H_\ze} &\lesssim \frac{\norm{f^{\sss (2)}}_{\H_\ze}}{d(\lm,\sigma( \K^0_{\ze}))} + \frac{\norm{\lap g^{\sss (1)}}_{\Hs_\ze}}{d(\lm,\sigma( \K^0_{\ze})) \dd d(\lee, - \sigma(\K^{0,*}_\ze))} \\
& + \frac{\norm{\lap f^{\sss (1)}}_{\Hs_\zm}}{d(\lm,\sigma( \K^0_{\ze})) \dd d(\lm, - \sigma(\K^{0,*}_\zm))}.
\end{align*}
Similarly, we can compute $\psi^{\sss (2)}$ and derive the estimate
\begin{align*}
\norm{\psi^{\sss (2)}}_{\H_\zm} &\lesssim \frac{\norm{g^{\sss (2)}}_{\H_\zm}}{d(\lee,\sigma(\K_\zm^0))} + \frac{\norm{\lap f^{\sss (1)}}_{\Hs_\zm}}{ d(\lm, - \sigma(\K^{0,*}_\zm)) \dd d(\lee,\sigma(\K_\zm^0))} \\
& + \frac{\norm{\lap g^{\sss (1)}}_{\Hs_\ze}}{d(\lee, - \sigma(\K^{0,*}_\ze)) \dd d(\lee,\sigma(\K_\zm^0))}.
\end{align*}
Now the above arguments conclude the uniquely solvability of
the system (\ref{in:theoprf:sys}), and the desired estimate (\ref{esti:norm:wob}).
\end{proof}

\begin{remark}
If we restrict the operator $\W_{0, B}$ on $\hh \t \hh$, then $\W_{0, B}$ has a diagonal form
\begin{equation*}
\W_{0,B} = \mm
\frac{\muc + 1}{2}I + (\muc - 1) \M_\zm^0 & 0 \\
0  &  k^2\frac{\eec+ 1}{2}I + k^2(\eec - 1)\M_\ze^0
\nn\,,
\end{equation*}
and it is an isomorphism on $\hh \t \hh$,  with the estimate
$\norm{\W_{0,B}^{-1}} \lesssim {1}/{d_\sigma}$.
\end{remark}

By the recurrence relation (\ref{for:Recursive relationship:green}) and the elliptic regularity, we conclude that $\W_{n,B}$
are uniformly bounded with respect to $n$, hence leading to the uniform operator convergence:
$$\lim_{\d \to 0} \W_{\d, B}^{-1} = \W^{-1}_{0,B}.$$
Therefore, there exists a $\d_0>0$ such that the following equivalence holds for $\d \le \d_0$:
\begin{equation*}
\norm{\W_{\d, B}^{-1}} \approx \norm{\W^{-1}_{0,B}}.
\end{equation*}
Combining this with Theorem\,\ref{thm: perp : W0B},
we observe directly that at some specified frequencies,
the norm of the solution operator $\norm{\W_{\d, B}^{-1}}$ may blow up with order
${(d_\sigma d_\sigma^*)^{-1}}$, which indicates the existence of resonances.
\subsection{Approximate scattered field}
We are now in a position to discuss how to approximate the scattered field with a certain order.
In view of the complexities and technicalities of the detailed computations and relevant estimates,
we split this section into three parts to make it more readable.
The main result of this section is given in Theorem \ref{thm:mainresult}.

\textbf{Approximate kernel and density.}
Motivated by the well-known two-scale asymptotic expansion method in the standard homogenization theory\cite{allaire1999boundary}, we shall first separate the propagative component from the scattered wave in the macroscopic scale. To this purpose, we observe from (\ref{quasi-green}) that the quasi-periodic Green's function $G^\k_{\emm}$ consists of
a propagating mode:
\begin{equation} \label{def:propagating mode}
    G^\k_{p,\emm}(\x,\y) = \frac{i}{2\tau k_3} e^{ik'\dd (x'-y')- ik_3|x_3-y_3|} \pm \frac{i}{2 \tau k_3}e^{ik'\dd (x'-y')- ik_3|x_3+y_3|},
\end{equation}
and an exponentially decaying mode: $G^\k_{e,\emm}:= G^\k_{\emm} - G^\k_{p,\emm}$.
Replacing $G^\k_{\emm}$ with $G^\k_{p,\emm}$ and $G^\k_{e,\emm}$, we can define $\Pi_{p,\emm}$ and $\Pi_{e,\emm}$ respectively as what we did in \eqref{def:kernel}. Therefore we can write
$\A^\k_{B,\emm}[\phi]  = \A^\k_{p,\emm}[\phi] + \A^\k_{e,\emm}[\phi]$, where
\begin{align*}
\A_{p,\emm}^\k[\phi] = \int_{\p B}  \Pi_{p,\emm}^\k(\x,\y)\phi(\y) d\sigma\,, \q
\A_{e,\emm}^\k[\phi] = \int_{\p B}  \Pi_{e,\emm}^\k(\x,\y)\phi(\y) d\sigma .
\end{align*}
We now define the propagative part and evanescent part of the scattered wave in the reference space respectively by
\begin{eqnarray}
&&\w{E}^r_{p}(\w{\x}) = \curl \A^{\d\k}_{p,\zm}[\w{\phi}](\w{\x}) + \frac{1}{\d} \curl \curl \A^{\d\k}_{p,\ze}[\w{\psi}](\w{\x}),
\label{for:def:erp} \\
&&\w{E}^r_{e}(\w{\x}) = \curl \A^{\d\k}_{e,\zm}[\w{\phi}](\w{\x}) + \frac{1}{\d} \curl \curl \A^{\d\k}_{e,\ze}[\w{\psi}](\w{\x}),
\end{eqnarray}
for densities $\w{\phi}$ and $\w{\psi}$ satisfying the system \eqref{eq:WdeltaB}.
We can see, from the definition of $G^\k_{e,\emm}$, that the structure of the evanescent wave is much more complicated than the one of the propagative part. Fortunately, we only care about the wave in the far field, where the effect of the evanescent wave can be ignored. Indeed, we have the following approximation estimate.
\begin{lemma}\label{lem:Erp}
Fix a constant $L\in \R^+$ such that $|\xd| < L$ for all $\x \in D$.
Then for small enough $\d$, there exists some positive constant
$c$ independent of $\d$ such that
$$
\sup_{\x \in \R^2 \t (L,+\infty)}|E^r(\x) - \w{E}^r_p(\frac{\x}{\d})| = O(\frac{1}{\d}e^{-\frac{c L}{\d}}).
$$
\end{lemma}
\begin{proof}
By the scaling property \eqref{for:asy:A} of $\A^\k_{D,\emm}$, we have
\begin{align} \label{proof:repre:ee}
    \w{E}^{r}(\w{\x}) - \w{E}^r_{p}(\w{\x}) = \w{E}^r_{e}(\w{\x}) =  \curl \A^{\d \k}_{e,{\rm m}}[\w{\phi}](\w{\x}) + \frac{1}{\d} \curl \curl \A^{\d \k}_{e,{\rm e}}[\w{\psi}](\w{\x}).
\end{align}
We now estimate these two terms. For large enough $\w{x}_3$, we
can separate the variable of the kernel $\Pi_{e,\emm}^{\d \k}$ involved in the definition of $\A^{\d \k}_{e,\emm}$:
\begin{equation} \label{def:sepker}
    \Pi_{e,\emm}^{\d \k}(\w{\x},\w{\y}) =
- \frac{1}{2 \tau} \sum_{\xi \in \Lambda^*\backslash \{0\}} \rho^{\d \k}_\xi(\w{\x}) \pi_{\xi,\emm}^{\d \k}(\w{\y})\,,
\end{equation}
where $\rho^{\d \k}_\xi(\w{\x})$ is given by
\begin{equation} \label{def:rhoker}
    \rho^{\d \k}_\xi(\w{\x}) :=  \frac{1}{\sqrt{|\xi+ \d k'|^2 - (\d k) ^2}}e^{i(\xi + \d k')\dd \w{\xb}}e^{-\sqrt{|\xi+ \d k'|^2 -(\d k)^2} (\w{x}_3 - h)}.
\end{equation}
Here $h$ is a constant satisfying $|\w{x}_3| < h < \frac{L}{\d}$ for $\w{\x} \in B$, and then $\pi_{\xi,\emm}^{\d \k}(\w{\y})$ can be introduced naturally and determined uniquely by \eqref{def:sepker} and \eqref{def:rhoker}.
We note that $\Pi_{e,\emm}^{\d \k}$ is a diagonal matrix, and its diagonal entries
are all smooth functions.
For the first term in (\ref{proof:repre:ee}), we can write
\begin{align} \label{eq:auxfir_1}
     \curl \A^{\d \k}_{e,{\rm m}}[\w{\phi}](\w{\x}) =\int_{\p B}\curl \Pi^{\d \k}_{e,\zm}(\w{\x},\w{\y})\w{\phi}(\w{\y})d\w{\y},
\end{align}
where
\begin{align} \label{eq:auxfir_2}
 \curl \Pi^{\d \k}_{e,\zm}(\w{\x},\w{\y}) =
    - \frac{1}{2 \tau} \sum_{\xi \in \Lambda^*\backslash \{0\}}  \gg \rho^{\d \k}_\xi(\w{\x}) \t \pi_{\xi,\zm}^{\d \k}(\w{\y}).
\end{align}
For $\rho^{\d \k}_\xi(\w{\x})$, we can see the existence of a positive constant $c$ such that the following estimate holds
for all $\xi \in \Lambda^*\backslash \{0\}$, uniformly with respect to all small enough $\d$:
$$|\p_j \rho^{\d \k}_\xi(\w{\x})| \lesssim e^{-c|\xi|(\w{x}_3 - h)}\,.$$
Using this together with the Cauchy's inequality and the trace inequality, we derive
\begin{align*}
   & |\int_{\p B} \gg \rho^{\d \k}_\xi(\w{\x}) \t \pi_{\xi,\zm}^{\d \k}(\w{\y}) \phi(\w{\y}) d\sigma(\w{\y})|
    \lesssim  e^{-c|\xi|(\w{x}_3-h)} \norm{\phi}_{\sss \hsn}  \sum^3_{j = 1}\norm{(\pi^{\d \k}_{\xi,\zm})_j}_{\sss \hst} \\
    \lesssim &  e^{-c|\xi|(\w{x}_3-h)} \norm{\phi}_{\sss \hsn}  \sum^3_{j = 1}\norm{(\pi^{\d \k}_{\xi,\zm})_j}_{\sss H^1(B)}
   \lesssim    e^{-c|\xi|(\w{x}_3-h)} \norm{\phi}_{\sss \hsn},
\end{align*}
where we have used the uniform boundedness of $\norm{(\pi^{\d k}_{\xi,\zm})_j}_{H^1(B)}$
with respect to $\xi \in  \Lambda^*\backslash \{0\}$. Now it follows easily from the above estimate and \eqref{eq:auxfir_1}-\eqref{eq:auxfir_2} that
\begin{equation*}
    |\curl \A^{\d k}_{e,{\rm m}}[\w{\phi}](\w{\x})| \lesssim e^{- c  (\w{x}_3 - h)} \norm{\phi}_{\hsn}.
\end{equation*}
Similarly, we can establish a desired pointwise estimate of the second term in \eqref{proof:repre:ee}.
Then the application of the above pointwise estimate of the two terms
in \eqref{proof:repre:ee} and a direct computation leads to the desired error estimate in Lemma\,\ref{lem:Erp}.
\end{proof}

To further derive a proper approximation of the propagative scattered wave $\w{E}^r_{p}$, it suffices to approximate the propagative kernel
$\Pi^{\d \k}_{p,\emm}$ and the two densities $\w{\psi}, \w{\phi}$. We consider the approximation of the integral kernel first, for which
we define two linear operators $\hat{\A}^{\d \k}_{p,\emm}[\w{\phi}]$ and $\app{\emm}^{\d \k}[\w{\phi}]$
for $\w{\phi} \in \htd$:
\begin{align*}
\hat{\A}^{\d \k}_{p,\emm}[\w\phi](\w{\x}) &:= \int_{\p B} [g_{\emm} \e_1, g_{\emm} \e_2, g_{\me} \e_3](\w{\y}) \w{\phi}(\w{\y})
d\sigma ,\\
\app{\emm}^{\d  \k}[\w{\phi}](\w{\x}) &:= e^{i\d \k^* \dd \w{\x}}
 \int_{\p B} [g_{\emm} \e_1, g_{\emm} \e_2, g_{\me} \e_3](\w{\y}) \w{\phi}(\w{\y}) d\sigma ,
\end{align*}
where $g_{\emm}(\w{\y})$ is given by 
\begin{equation*}
    g_\ze(\w{\y}) = - \frac{\w{y}_3}{\tau}\,, \q g_\zm(\w{\y}) = \frac{i}{\tau \d k_3} + \frac{d'\dd \w{y}'}{d_3 \tau}.
\end{equation*}
We can see that $e^{i \d  \k^*\dd \w{\x}} g_{\emm}(\w{\y})$ are
 good approximations of $G^{\d \k}_{p,\emm}$ (cf.\eqref{def:propagating mode}), by noting the fact that
 $$\frac{i e^{-i \d \k^* \dd \w{\y}}}{2 \tau \d k_3} = \frac{i}{2\tau \d k_3} +  \frac{\di^*\dd \w{\y}}{2 d_3\tau} + O(\d).$$
Then by a direct verification, we have the following estimate.
 %
\begin{lemma} \label{app:reflect}
It holds for small enough $\d$ and $\xd \ge L$ (the constant given in Lemma\,\ref{lem:Erp}), and all $\w{\phi}, \w{\psi} \in  \htd$ that
\begin{align*}
|\curl (\A_{p,\zm}^{\d \k} - \app{\zm}^{\d  \k})[\w{\phi}]|(\w{\x}) + \frac{1}{\d}|\curl \curl (\A_{p,\ze}^{\d \k} - \app{\ze}^{\d  \k})[\w{\psi}]|(\w{\x}) \lesssim \d^2  \Big\{\norm{\w{\phi}}_{\htd}  + \norm{\w{\psi}}_{\htd}\Big\}.
\end{align*}
\end{lemma}
To proceed our approximation, we define
the Green's tensor $\G_{r}^{\d \k}(\w{\x})$ associated with the propagative mode $g_r^{\d \k}(\w{\x}) = e^{i\d\k^*\dd\w{\x}}$:
\begin{align}
    \G_r^{\d \k}(\w{\x})  = g_r^{\d \k}(\w{\x})\I + \frac{1}{\d^2 k^2}\gg^2 g_r^{\d \k}(\w{\x})
     =  (\I - \di^* \otimes \di^*)g_r^{\d \k}(\w{\x}). \label{eq:proker}
\end{align}
We remark that the matrix $\I - \di^* \otimes \di^*$ is the projection on the orthogonal complement of the linear space spanned by $\di^*$.
Using this, we can directly check that
\begin{align*}
    \curl \A^{\d \k}_{p,\zm,0}[\w{\phi}] = \curl \G_r^{\d \k}
     \hat{\A}^{\d \k}_{p,\zm}[\w{\phi}], \q
     \curl \curl \A^{\d \k}_{p,\ze,0}[\w{\psi}] = (\d k)^2 \G_r^{\d \k}\hat{\A}^{\d \k}_{p,\ze}[\w{\phi}],
\end{align*}
which, together with (\ref{for:def:erp}) and Lemma \ref{app:reflect}, results in the asymptotic expansion
\begin{equation} \label{app:reflect:err}
    \w{E}^r_p =  \curl \G_r^{\d \k} \hat{\A}^{\d \k}_{p,\zm}[\w{\phi}]  + \d k^2 \G_r^{\d \k}\hat{\A}^{\d \k}_{p,\ze}[\w{\phi}] + O(\d^2).
\end{equation}

Now we intend to work out the leading-order terms in  the densities $\w{\phi}$ and $\w{\psi}$.
Using the Taylor expansion of the incident wave
\begin{equation*}
\mm
\n(\w{\x}) \t \w{E}^i(\w{\x})\\ ik \n(\w{\x}) \t \w{H}^i(\w{\x})
\nn = \sum_\beta \frac{\d^{|\beta|}}{\beta !} \mm \n(\w{\x}) \t \w{\x}^\beta \p^\beta E^i(0) \\ ik \n(\w{\x}) \t \w{\x}^\beta \p^\beta H^i(0)  \nn\,,
\end{equation*}
we can write
\begin{align} \label{for:expan:density1}
\w{\phi} = \sum_{\beta } \frac{\d^{|\beta|}}{\beta!}\w{\phi}_\beta, \q  \w{\psi} = \sum_\beta \frac{\d^{|\beta|}}{\beta!}\w{\phi}_\beta,
\end{align}
by setting that
\begin{equation} \label{for:equ:phibeta}
\W_{\d,B} \mm  \w{\phi}_\beta \\ \w{\psi}_\beta \nn(\w{\x}) = \mm \n(\w{\x}) \t \w{\x}^\beta \p^\beta E^i(0) \\ ik \n(\w{\x}) \t \w{\x}^\beta \p^\beta H^i(0)  \nn.
\end{equation}
We should note from (\ref{for:equ:phibeta}) that $\phi_\beta$ and $\psi_\beta$ still depend on $\d$.
Indeed, recalling the expansion
$$\W_{\d,B} = \sum^\infty_{n = 0} \d^n \W_{n,B} =: \W_{0,B} - \d \W_{r,B}, $$
when ${\d}/{(d_\sigma d^*_\sigma)}$ is small enough, we can expand $\W^{-1}_{\d,B}$ in Neumann series
\begin{align}
    \W^{-1}_{\d,B} = (I - \d \W^{-1}_{0,B} \W_{r,B})^{-1} \W^{-1}_{0,B} = \sum^\infty_{n = 0} \d^n (\W^{-1}_{0,B}\W_{r,B})^n \W^{-1}_{0,B}.
\end{align}
This shows that
\begin{equation} \label{for:expan:density2}
    \w{\phi}_{\beta} = \sum^\infty_{j = 0} \d^j \w{\phi}_{\beta,j} \q and \q  \w{\psi}_{\beta} = \sum^\infty_{j = 0} \d^j \w{\psi}_{\beta,j},
\end{equation}
where $\w{\phi}_{\beta,j}$ and $\w{\psi}_{\beta,j}$ are determined by
\begin{align}
    \mm \w{\phi}_{\beta,j} \\ \w{\psi}_{\beta,j} \nn = (\W^{-1}_{0,B}\W_{r,B})^j\W^{-1}_{0,B} \mm \n(\w{\x}) \t \w{\x}^\beta \p^\beta E^i(0) \\ ik \n(\w{\x}) \t \w{\x}^\beta \p^\beta H^i(0) \nn.
\end{align}
For simplicity, we write $\w{\phi}_{\beta,0}$ with $|\beta| = 1$ below as $\w{\phi}_{j,0}$ for $j=1,2,3$.
Then it follows from expansions (\ref{for:expan:density1}) and (\ref{for:expan:density2}) that
\begin{align}
    \w{\phi} = \w{\phi}_{0,0} + \d \w{\phi}_{0,1} + \d \sum^3_{j = 1}\w{\phi}_{j,0} + O(\d^2)\,, \q
    \w{\psi} = \w{\psi}_{0,0} + \d \w{\psi}_{0,1} + \d \sum^3_{j = 1}\w{\psi}_{j,0} + O(\d^2).
\end{align}
The error terms are measured in $\htd$.
Substituting these expansions
into the approximate scattered field (\ref{app:reflect:err}), we can further write
\begin{align}
    \w{E}^r_{p}(\w{\x}) = & \curl \G_{r}^{\d \k} (\int_{\p B}\frac{i + \d k d' \dd \w{y}' }{\tau \d k_3} (\w{\phi}',0)^t  d\sigma - \int_{\p    B}\frac{\w{y}_3}{\tau}\w{\phi}_3\e_3 d\sigma)  \notag  \\
    & +  \d k^2 \G_r^{\d \k} (\int_{\p B}-\frac{\w{y}_3}{\tau}(\w{\psi}',0)^t  d\sigma + \int_{\p B} \frac{i + \d k d' \dd \w{y}'}{\tau \d k_3}\w{\psi}_3\e_3  d\sigma) + O(\d^2) \notag \\
    = & \curl \G_r^{\d \k}( \frac{i}{\tau  k_3}\int_{\p B} (\w{\phi}'_{0,1} + \sum_{j =1}^3 \w{\phi}'_{j,0},0)^t  d\sigma + \int_{\p B}\frac{d' \dd \w{y}'}{\tau d_3}(\w{\phi}'_{0,0},0)^t  d\sigma) - \int_{\p B}\frac{\w{y}_3}{\tau} (\w{\phi}_{0,0})_3\e_3  d\sigma)   \notag  \\
     & +  \d k^2 \G_r^{\d \k} (\int_{\p B}-\frac{\w{y}_3}{\tau}(\w{\psi}'_{0,0},0)^t  d\sigma + \int_{\p B} \frac{ d' \dd \w{y}'}{\tau d_3}(\w{\psi}_{0,0})_3\e_3  d\sigma  \notag \\
      & +   \frac{i}{\tau k_3} \int_{\p B} (\w{\psi}_{0,1})_3\e_3 + \sum_{j =1}^3 (\w{\psi}_{j,0})_3\e_3 d\sigma +  O(\d^2)\,,
      \label{appro:reflection}
\end{align}
where the superscript $t$ denotes the transport of a vector.

\textbf{Computing the leading-order densities.}
We readily see from \eqref{for:equ:phibeta} that the zero order terms $\w{\phi}_{\beta,0}$ and $\w{\psi}_{\beta,0}$
in (\ref{for:expan:density2}) satisfy the following equation:
\begin{equation*}
\W_{0,B}
\mm
\w{\phi}_{\beta,0} \\
\w{\psi}_{\beta,0}
\nn(\w{\x}) =
\mm
\n(\w{\x}) \t \w{\x}^\beta \p^\beta \ei(0) \\
i k \n(\w{\x}) \t \w{\x}^\beta \p^\beta \hi(0)
\nn,
\end{equation*}
which has already been studied in (\ref{in:theoprf:sys}), with the solutions given by
\begin{equation}  \label{form:potenzero:1}
    \w{\phi}_{\beta,0} =(\lm - \M_\zm)^{-1} (\frac{\n(\w{\x})\times \w{\x}^\beta \p^\beta \ei(0)}{1 -\muc} + f_\beta),
\end{equation}
\begin{equation} \label{form:potenzero:2}
    \w{\psi}_{\beta,0} = (\lambda_\varepsilon - \M_\ze)^{-1}(\frac{i\n(\w{\x})\times \w{\x}^\beta \p^\beta \hi(0) }{k(1-\varepsilon_c)} + g_\beta),
\end{equation}
where $f_\beta$  and $g_\beta$ are defined by
\begin{equation*}
f_\beta :=  \frac{k - k_c}{1 -\muc} \L_{1,\ze}[\w{\psi}_{\beta,0}], \q g_\beta := \frac{k - k_c}{k^2(1- \varepsilon_c )}\L_{1,\zm}[\w{\phi}_{\beta,0}].
\end{equation*}
In particular, when $\beta = 0$, we know that $\w{\phi}_{0,0}$ and $\w{\psi}_{0,0}$ are divergence-free
using the facts that $\divs (\n \t \ei(0)) = 0$ and $\divs (\n \t \hi(0)) = 0$.
Further, the first-order terms $\w{\phi}_{0,1}$ and $\w{\psi}_{0,1}$ can be determined by
\begin{equation*}
\W_{0,B}
\mm
\w{\phi}_{0,1} \\
\w{\psi}_{0,1}
\nn + \W_{1,B}
\mm
\w{\phi}_{0,0} \\
\w{\psi}_{0,0}
\nn  = 0.
\end{equation*}
More precisely, this can be written componentwise as
\begin{align}
& (\lm - \M_\zm) [\w{\phi}_{0,1}] +  \frac{k_c -k}{1 -\muc} \L_{1,\ze}[\w{\psi}_{0,1}] + \frac{k_c - k}{1 - \muc}
\M_{1,\zm}[\w{\phi}_{0,0}] + \frac{k_c^2-k^2}{1-\muc}
\L_{2,\ze}[\w{\psi}_{0,0}] = 0, \label{Equ:potential one}\\
&\frac{k_c - k}{k^2(1-\varepsilon_c)}\L_{1,\zm}[\w{\phi}_{0,1}] + (\lambda_\varepsilon - \M_\ze)[\psb{1}] + \frac{k_c^2-k^2}{k^2(1-\eec)}\L_{2,\zm}[\pb{0}] + \frac{\eec k_c - k}{1-\eec}\M_{1,\ze}[\psb{0}] = 0. \label{Equ:potential one+}
\end{align}
We can see that $\w{\phi}_{0,1}$ and $\w{\psi}_{0,1}$ can be completely
determined by the above equations once $\w{\phi}_{0,0}$ and $\w{\psi}_{0,0}$ are solved. But noting
\begin{equation} \label{rela:useinte}
    \int_{\p B}\w{\phi}(\w{\y})  d\sigma = - \int_{\p B}\w{\y} \divs \phi(\w{\y})  d\sigma,
\end{equation}
we know that it suffices to find the surface divergence of $\w{\phi}_{0,1}$ and $\w{\psi}_{0,1}$
in order to compute (\ref{appro:reflection}).
We thus take the surface divergence on both sides of equations \eqref{Equ:potential one}
and \eqref{Equ:potential one+} to deduce that
\begin{align}
& (\lm + \K_\zm^{0,*})[\divs \w{\phi}_{0,1}] = \frac{k_c^2 - k^2}{1 - \muc} \gamma_n( \curl
\A_\ze[\w{\psi}_{0,0}]), \label{form:surdiv:phi} \\
& (\lee + \K_\ze^{0,*})[\divs \w{\psi}_{0,1}] =
\frac{k_c^2 - k^2}{k^2(1-\eec)}\gamma_n(\curl \A_\zm[\phi_{0,0}]), \label{form:surdiv:psi}
\end{align}
by using Lemma \ref{rela:all:ml}.
To facilitate our further computings, we follow \cite[Lemma 5.5]{ammari2016surface}
and introduce two harmonic systems with appropriate interface conditions to represent the quantities
$\curl \A_\ze[\w{\psi}_{0,0}]$ and $\curl \A_\zm[\w{\phi}_{0,0}]$
involved in \eqref{form:surdiv:phi} and \eqref{form:surdiv:psi} in terms of gradients:
\begin{equation} \label{sys:aux:ele}
    \begin{cases}
    \Delta u = 0,   \q & \text{in} \q  \Omega , \\
    (\n \dd \nabla u)|_- = (\n \dd \nabla u)|_+,  \q & \text{on} \q \p B , \\
   \muc (\n \t \gg u)|_- - (\n \t \gg u)|_+ =
    \n \t
    \ei(0), \q & \text{on} \q  \p B , \\
    u - u_\infty \ \text{is exponentially decaying}, \q &\text{as} \q x_3 \to \infty \\
    u = 0, \q  &\text{on} \q \Sigma \\
    \text{$u$ satisfies the periodic boundary condition} & \text{on} \q \p \Omega \backslash \Sigma,
    \end{cases}
\end{equation}
and
\begin{equation} \label{sys:aux:mag}
    \begin{cases}
    \Delta v = 0,   \q &\text{in} \q  \Omega , \\
    (\n \dd \nabla v)|_- = (\n \dd \nabla v)|_+,  \q &\text{on} \q \p B , \\
    \eec(\n \t \gg v)|_- - (\n \t \gg v)|_+ =
    \frac{i}{k} \n \t
    \hi(0), \q  &\text{on} \q  \p B , \\
    v  \ \text{is exponentially decaying}, \q &\text{as} \q x_3 \to \infty , \\
    \frac{\p v}{\p x_3} = 0, \q &\text{on} \q \Sigma , \\
    \text{$v$ satisfies the periodic boundary condition} & \text{on} \q \p \Omega \backslash \Sigma,
    \end{cases}
\end{equation}
where $u_\infty$ is a complex constant, and the solutions to these two systems
are denoted by $u^e$ and $u^h$, respectively.
The solutions to these two systems may not necessarily be unique,
but their gradients can be uniquely determined, as shown in the following lemma.
\begin{lemma} \label{lem:interpret:zeroorder}
For $\curl \A_\zm[\w{\phi}_{0,0}]$ and $\curl \A_\ze[\w{\psi}_{0,0}]$, it holds that,
\begin{align}  \label{lemma:aux:elec}
\gg u^e  =  \curl \A_\zm^0[\w{\phi}_{0,0}]  =
\begin{cases}
\frac{1}{1-\muc}\gg \S_\ze^0 (\lm - \K_\ze^{0,*})^{-1}[\n \dd \ei(0)]) \q &\text{in} \q \Omega \backslash B ,\\
\frac{1}{\muc}\ei(0) + \frac{1}{\muc(1-\muc)}\gg \S_e^0 (\lm - \K_\ze^{0,*})^{-1}[\n \dd \ei(0)]) \q &\text{in} \q B,
\end{cases}
\end{align}
and
\begin{align}  \label{lemma:aux:mag}
\gg u^h  = \curl \A_\ze^0 [\w{\psi}_{0,0}]
 =  \begin{cases}
\frac{i}{k(1-\eec)} \gg \S_\zm^0 (\lee - \K_\zm^{0,*})^{-1}[\n \dd \hi(0)] \q & \text{in} \q \Omega \backslash B ,\\
\frac{i}{k \eec} \hi(0) + \frac{i}{k\eec(1-\eec)}\nabla \S_\zm^0(\lee - \K_\zm^{0,*})^{-1}[\n \dd \hi(0)] \q & \text{in} \q B.
\end{cases}
\end{align}
\end{lemma}

\begin{proof}
We consider only the system \eqref{sys:aux:ele} and show \eqref{lemma:aux:elec}, since the proof of (\ref{lemma:aux:mag}) is similar. We first prove that the right-hand side of \eqref{lemma:aux:elec} and $\curl \A_\zm^0[\w{\phi}_{0,0}]$
are both the gradients of  some solutions to \eqref{sys:aux:ele},
then demonstrate that all the solutions to \eqref{sys:aux:ele} have an identical gradient.
Recalling the far-field behavior \eqref{far-field:single}-\eqref{far-field:single2} of $\S^0_{\ze}[\phi]$,
we can verify that the function
\begin{align*}
u(\w{\x}) :=
\begin{cases}
\frac{1}{1-\muc} \S_\ze^0 (\lm - \K_\ze^{0,*})^{-1}[\n \dd \ei(0)])(\w{\x}) \q &\text{in} \q \Omega \backslash B, \\
\frac{1}{\muc}\ei(0)\w{\x} + \frac{1}{\muc(1-\muc)} \S_\ze^0 (\lm - \K_\ze^{0,*})^{-1}[\n \dd \ei(0)])(\w{\x}) \q &\text{in} \q B ,
\end{cases}
\end{align*}
satisfies both the boundary and far-field conditions in (\ref{sys:aux:ele}), and is actually a solution to (\ref{sys:aux:ele}).
Furthermore, we can check that the right-hand side of (\ref{lemma:aux:elec}) is the gradient of this solution $u$.

Next, we show that $\curl \A_\zm^0 [\w{\phi}_{0,0}]$ can also be written as the gradient of a solution to (\ref{sys:aux:ele}).
In fact, by the continuity of its normal trace and the jump relation of its tangential trace,
we find that
\begin{equation*}
    [\n \dd \curl \A^0_\zm[\w{\phi}_{0,0}]] = 0\,, \q [\mu(\n \t \curl \A^0_\zm[\w{\phi}_{0,0}])] = \n \t \ei(0).
\end{equation*}
But noticing that $\divs \w{\phi}_{0,0} = 0$, we get
\begin{equation*}
    \curl \curl \A_\zm^0 [\w{\phi}_{0,0}] =  \gg \S_\zm^0[\divs \w{\phi}_{0,0}] = 0 \q \text{in} \q \Omega \backslash \p B,
\end{equation*}
which implies (cf. \cite[Theorem 3.37]{monk2003finite})
\begin{equation} \label{appen:poincare}
    \curl \A_\zm^0 [\w{\phi}_{0,0}] = \gg p \q \text{for some} \  p \in H^1(\B) \ \text{or} \ H^1_{loc}(\R^3_+ \backslash \bar{\B}).
\end{equation}
Moreover, we can assume $p = 0$ on $\Gamma$ by noting the fact that $\e_3 \t \curl \A^0_{\zm}[\w{\phi}_{0,0}] = \e_3 \t \gg p = 0$ on $\Gamma$.
To see that $p$ is indeed a solution to \eqref{sys:aux:ele}, it remains to show that $p$, up to a constant, satisfies
\begin{enumerate}[\textbullet]
    \item[(i)] $p$ is periodic with respect to $\Lambda$.
    \item[(ii)] there exists a complex constant $c_p$ such that $p - c_p$ decays exponentially as $x_3 \to \infty$.
\end{enumerate}
For these two claims, we consider the translation operator $\T_i: L_{loc}^2(\R^3_+) \longmapsto L_{loc}^2(\R^3_+)$ defined by
\begin{equation*}
    \T_i u (x',x_3) = u (x'+\a_i,x_3).
\end{equation*}
Note that $\T_i$ commutes with the gradient operator, namely $\T_i \gg = \gg T_i$ in the distribution sense. Since $\T_i \gg p = \gg p$ by (\ref{appen:poincare}), we have $\gg (\T_i p - p) = 0$, which implies that there exist two constants $C_1$ and $C_2$
such that $\T_i p = p + C_i$ in $\Omega \backslash \bar{B}$. We now choose vectors $\b_i$ such that $\a_i \dd \b_j = \d_{ij}$,
and define an auxiliary function $\w{p} = p - (C_1\b_1 + C_2 \b_2)\dd x'$.
Then we can directly check that $\w{p}$ is periodic with respect to $\Lambda$, i.e.,
\begin{align*}
    \w{p}(x'+\a_i,x_3) =  p(x'+\a_i,x_3) - (C_1\b_1 + C_2 \b_2)\dd (x'+ \a_i) = \w{p}(x',x_3).
\end{align*}
In the case of far-fields, noting that $\Delta \w{p} = 0$, we can expand $\w{p}$ by Fourier series  (cf.\cite{cessenat1996mathematical}),
\begin{equation} \label{far:auxwp}
    \w{p} = \sum_{\xi \in \Lambda^*} p_\xi e^{i\xi \dd x'-|\xi|x_3}, \q p_\xi \in \C.
\end{equation}
It is easy to see from (\ref{far-field:single})-(\ref{far-field:single2})
that $\curl \A^0_{\zm}[\w{\phi}_{0,0}]$ decays exponentially as $x_3 \to \infty$. Recalling \eqref{appen:poincare}-\eqref{far:auxwp} and the definition of $\w{p}$, we can show that  $C_1 = 0$ and $C_2 = 0$,
by matching the far-field modes of $\gg p$ and $\curl \A_\zm^0[\w{\phi}_{0,0}]$.
Hence we can conclude $\w{p} = p$, and our two claims follow.

Finally, we prove $\gg u$ and $\gg p$ defined above can be uniquely determined by the system (\ref{sys:aux:ele}). For doing so, it suffices to show that the gradient of any solution $u^e$ to \eqref{sys:aux:ele} is zero in $\Omega \backslash \p B$
if we replace the jump data $\n \t \ei(0)$ by $0$.
Noting that the jump condition $\muc (\n \t \gg u^e)|_- = (\n \t \gg u^e)|_+$, together with the formula (\ref{relation:trace:l1}), implies that
$$\gs [(\mu u^e)|_- -  (\mu u^e)|_+] = 0,$$
therefore we know $(\mu u^e)|_- = (\mu u^e)|_+ + C$ for some constant $C$.
Without loss of generality, we assume $C=0$, otherwise we may consider $u^e - \frac{C}{\muc} \mathcal{X}_B$.
By integration by parts and interface conditions, we get
\begin{equation} \label{4.33}
     \int_{\Omega} \mu |\gg u^e|^2(\w{\x}) d\w{\x} = 0\,,
\end{equation}
then taking the imaginary and real parts, we deduce
$\gg u^e = 0$ in $B$, and $\gg u^e = 0$ in $\Omega \backslash \bar{B}$, respectively.
\end{proof}
We can see from Lemma\,\ref{lem:interpret:zeroorder} and the formulas (\ref{form:surdiv:phi})-(\ref{form:surdiv:psi})
that
\begin{equation} \label{for:diver:firstorder}
\divs \w{\phi}_{0,1} = (\lm + \K_\zm^{0,*})^{-1} ( \frac{k_c^2-k^2}{1-\muc} \pn{u^h})\,, \q
\divs \w{\psi}_{0,1} =  (\lee +
\K_\ze^{0,*})^{-1}(\frac{k_c^2-k^2}{k^2(1-\eec)} \pn{u^e}).
\end{equation}
In order to calculate $\w{E}^r_{p}$ in \eqref{appro:reflection}, we still need to find  quantities like
$\int_{\p B}\w{\phi}_{j,0}d\sigma$, $\int_{\p B}\w{\phi}_{0,j}d\sigma$ and $\int_{\p B}\w{y}_j\w{\phi}_{0,0}d\sigma$,
and the corresponding quantities for $\w{\psi}$, which are given in the following lemma.

\begin{lemma}\label{lem:tensor:firstorder}
The following identities hold,
\begin{align}
&\int_{\p B} \w{y}_j \w{\phi}_{0,0}(\w{\y}) d\sigma
 = |B|\e_j \t \ei(0) + (1 - \muc) \e_j \t \int_{\p B} \w{\y} \pn{u^e}(\w{\y}) d\sigma\,, \label{for:tensor:1}\\
&\int_{\p B} \w{y}_j \w{\psi}_{0,0}(\w{\y}) d\sigma =   \frac{i}{k}|B|\e_j \t \hi(0) + (1-\eec) \e_j \times \int_{ \p B} \w{\y} \pn{u^h}(\w{\y}) d\sigma\,,\label{for:tensor:2}\\
&\int_{\p B} \w{\phi}_{j,0}(\w{\y}) d\sigma
= \e_j \times \p^j \ei(0) |B|
 + (1 - \muc)\int_B \gg \S^0_\zm[\divs \w{\phi}_{j,0}](\w{\y})d\w{\y}\,, \label{for:tensor:3}\\
&\int_{\p B} \w{\psi}_{j,0}(\w{\y}) d\sigma = \frac{i}{k}\e_j \t \p^j \hi(0) |B| + (1 -\eec) \int_B \gg \S^0_\ze[\divs \w{\psi}_{j,0}](\w{\y}) d\w{\y}\,.\label{for:tensor:4}
\end{align}
\end{lemma}
\begin{proof}
We demonstrate only how to compute the quantities involving $\w{\psi}$, as the same can be done
for the terms related to $\w{\phi}$.
To do so, we first consider $\int_{\p B} \w{\y}^\alpha \w{\psi}_{\beta,0}(\w{\y}) d\sigma$ for $|\alpha|\le 1$.
Using  formula (\ref{form:potenzero:2}) and the decomposition for any proper $f$,
\begin{align*}
f = (1 - \eec)(\lambda_\varepsilon - \M_\ze + \frac{1}{2} + \M_\ze)[f]\,,
\end{align*}
we can compute
\begin{align*}
\int_{\p B} \w{\y}^\alpha \w{\psi}_{\beta,0}(\w{\y}) d\sigma = & \int_{\p B}\w{\y}^\alpha (1 - \eec)(\lambda_\varepsilon - \M_\ze + \frac{1}{2} + \M_\ze)[\w{\psi}_{\beta,0}](\w{\y}) d\sigma \\
 = &\int_{\p B}\w{\y}^\alpha (1 - \eec)(\frac{i\n(\w{\y})\times \w{\y}^\beta \p^\beta \hi(0) }{k(1-\varepsilon_c)} + g_\beta) d\sigma \\
& + \int_{\p B} \w{\y}^\alpha (1-\eec) \n \t \curl \A_\ze[\w{\psi}_{\beta,0}] d\sigma \\
= &\frac{i}{k}\int_{B} \curl (\w{\y}^\alpha \w{\y}^{\beta} \p^\beta \hi(0))d\w{\y} + (1-\eec)\int_{\p B}\w{\y}^\alpha g_\beta d\sigma \\
& +  (1 - \eec)\int_{B} \curl (\w{\y}^\alpha \curl \A_\ze[\w{\psi}_{\beta,0}])d\w{\y}.
\end{align*}
In particular, we get for $|\alpha| = 1\,, |\beta| = 0$ that
\begin{align*}
    \int_{\p B} \w{y}_j \w{\psi}_{0,0} d\sigma = \frac{i}{k}\e_j \t \hi(0)|B| + (1-\eec)\int_{B}  \curl (\w{y}_j \curl \A_e[\w{\psi}_{0,0}])
    d\w{\y},
\end{align*}
where we have used the Stokes's theorem and the fact that $g_\beta= 0$ for $\beta = 0$.
Then formula (\ref{for:tensor:2}) follows directly from the relation
\begin{align*}
    &\curl (\w{y}_j \curl \A_\ze[\w{\psi}_{\beta,0}])
 =   \e_j \t \curl \A_\ze[\w{\psi}_{0,0}] + \gg \S_\ze^0[\divs \w{\psi}_{j,0}],
\end{align*}
and by writing
\begin{align*}
    &\int_{B} \curl \A_\ze[\w{\psi}_{0,0}] d\w{\y} = \int_{\p B} (\n \t \A_\ze)|_-[\w{\psi}_{0,0}]d\sigma \\
= & \int_{\p B}(\frac{1}{2}+\M_\ze)[\w{\psi}_{0,0}]d\sigma = - \int_{\p B}\w{\y} \divs(\frac{1}{2}+\M_\ze)[\w{\psi}_{0,0}]d\sigma \\
= & - \int_{\p B} \w{\y} (\frac{1}{2} - \K^*_\ze)[\divs \w{\psi}_{0,0}]d\sigma  =  0\,,
\end{align*}
where we have used \eqref{rela:useinte} again.
For $|\alpha|= 0,|\beta| = 1$, we note that
$$\int_{\p B} g_\beta d\sigma = 0 \q for \q |\beta| = 1, $$
then a similar derivation leads to (\ref{for:tensor:4}).
\end{proof}

\textbf{Computation of the scattered wave.} We are now well prepared to compute each term
in (\ref{appro:reflection}). Recalling our conventional writing $  \R^3  \ni  \di= (d', d_3)$
for a vector $\di$, we identify $d'$ with $(d',0)$ below to simplify our notation. We start with a direct application of Lemma \ref{lem:tensor:firstorder} to get
\begin{align} \label{aux:firstorder:1}
     \int_{\p B}\frac{d'\dd \w{y}' }{\tau d_3}\w{\phi}_{0,0}(\w{\y})d\sigma = \frac{d' \t \ei(0)|B|}{\tau d_3} + \frac{d'}{\tau d_3} \t \int_{\p B}(1-\muc) \w{\y} \pn{u^e}(\w{\y}) d\sigma,
\end{align}
and
\begin{align}
        -\frac{1}{\tau}\int_{\p B} \w{y_3} \w{\psi}_{0,0}(\w{\y})d\sigma &= -\frac{i}{k \tau} |B| \e_3 \t \hi(0) + \frac{\eec -1}{\tau} \e_3 \t \int_{\p B}\w{\y}\pn{u^h}(\w{\y}) d\sigma , \label{aux:firstorder:2} \\
          \int_{\p B} \frac{d' \dd \w{y}'}{\tau d_3}\w{\psi}_{0,0}(\w{\y}) d\sigma &= \frac{i}{\tau k_3}|B| d' \t \hi(0) + \frac{1-\eec}{\tau d_3} d' \t \int_{\p B} \w{\y} \pn{u^h}(\w{\y})d\sigma .\label{aux:firstorder:3}
\end{align}
We remark that it is unnecessary for us to consider $\int_{\p B}\frac{\w{y_3}}{\tau} \w{\phi}_{0,0}(\w{\y})d\sigma$
since only the third component is needed in (\ref{appro:reflection}), which is known to be zero.
It is easy to check that
\begin{equation}
    \sum_{j = 1}^3 \e_j \t \p^j \ei(0)  = ik \hi(0) \q \text{and} \q \sum^3_{j =1} \e_j \t \p^j\hi(0) = -ik \ei(0), \label{aux:constant}
\end{equation}
then we can derive
\begin{align}
    \sum^3_{j =1} \divs \w{\phi}_{j,0} &= \sum^3_{j =1} (\lm + \K_\zm^{0,*})^{-1 }(\frac{\divs \n(\w{\x}) \t \w{x}_j\p^j \ei(0) }{1-\muc}) \notag \\
    & = \sum^3_{j =1} (\lm + \K_\zm^{0,*})^{-1 }(\frac{\n(\w{\x}) \dd (\e_j \t \p^j \ei(0))}{\muc -1}) \notag \\
    & = \frac{i k}{\muc - 1}(\lm + \K^{0,*}_\zm)^{-1}[\n \dd \hi(0)]. \label{aux:sum:divsphi}
\end{align}
Using (\ref{aux:constant}) and (\ref{aux:sum:divsphi}),  we can obtain the summation of (\ref{for:tensor:3}) over $j$:
\begin{align}
     \frac{i}{\tau  k_3}\int_{\p B}\sum_{j =1}^3 \w{\phi}_{j,0}d\sigma & = - \frac{\hi(0)|B|}{\tau d_3} + \frac{1}{\tau d_3} \int_{B} \gg \S^{0,*}_{\zm}(\lm +  \K^{0,*}_\zm)^{-1}[\n \dd \hi(0)] \notag \\
     & = - \frac{\hi(0)|B|}{\tau d_3} +   \frac{1}{\tau d_3} \int_{\p B}\w{\y} (-\frac{1}{2} + \K_\zm^{0,*})(\lm + \K_\zm^{0,*})^{-1}[\n \dd \hi(0)] d\sigma  \notag \\
    & = - \frac{\hi(0)|B|}{\tau d_3} +  \frac{1}{\tau d_3(\muc -1)} \int_{\p B} \w{\y} (\lm + \K_\zm^{0,*})^{-1}[\n \dd  \hi(0)] d\sigma. \label{aux:firstorder:4}
\end{align}
A similar calculation gives
\begin{align}
   \frac{i}{\tau k_3}\int_{\p B} \w{\psi}_{j,0} d\sigma  = & \frac{i}{\tau k_3}\ei(0)|B| -  \frac{i}{\tau k_3}\int_{B}\gg \S_\ze^0(\lee + \K_\ze^{0,*})^{-1}[\n \dd \ei(0)]d\sigma \notag \\
    = & \frac{i}{\tau k_3}\ei(0)|B| +  \frac{i}{\tau k_3(1-\eec)}\int_{B} \w{\y} (\lee + \K_\ze^{0,*})^{-1}[\n \dd \ei(0)]d\sigma, \label{aux:firstorder:5}
\end{align}
by using (\ref{aux:constant}) and the fact that
\begin{equation}
    \sum^3_{j = 1} \divs \w{\psi}_{j,0} = \frac{1}{\eec - 1}(\lee + \K_\ze^{0,*})^{-1}[\n \dd \ei(0)].  \label{aux:sum:divspsi}
\end{equation}
Moreover, recalling (\ref{for:diver:firstorder}), we get
\begin{align}
    & \frac{i}{\tau  k_3}\int_{\p B} \w{\phi}_{0,1} d\sigma = \frac{i}{\tau k_3}\int_{\p B}
    \w{\y}(\lm + \K_\zm^{0,*})^{-1}(\frac{k^2 -  k_c^2}{1-\muc} \pn{u^h})(\w{\y})d\sigma \label{aux:zeroorder:1}, \\
    & \frac{i}{\tau k_3} \int_{\p B} \w{\psi}_{0,1} d\sigma = - \frac{i}{\tau k_3} \int_{\p B} \w{\y} \frac{k_c^2 - k^2}{k^2(1-\eec)} (\lee + \K^{0,*}_\ze)^{-1}(\pn{u^e})(\w{\y}) d\sigma. \label{aux:zeroorder:2}
\end{align}
We have now computed all the terms involved in (\ref{appro:reflection}).
It is worth mentioning that we shall only need the first two components of (\ref{aux:firstorder:1})-(\ref{aux:firstorder:2}),(\ref{aux:firstorder:4}) and (\ref{aux:zeroorder:1}), as well as the third component of (\ref{aux:firstorder:3}), (\ref{aux:firstorder:5}) and (\ref{aux:zeroorder:2}), to compute the approximate scattered wave.
Before we apply all the expressions to (\ref{appro:reflection}), we make some further observations
to simplify our  representation.
To proceed, we first consider the non-integral terms in (\ref{aux:firstorder:1}), (\ref{aux:firstorder:2}) and  (\ref{aux:firstorder:4})
to find that
\begin{align*}
    & \curl \G_r^{\d \k} (\frac{-H^i(0)|B|}{\tau d_3} + \frac{|B|d' \t \ei(0)}{\tau d_3}) + \d k^2 \G_r^{\d \k} (-\frac{i}{k \tau}|B|\e_3 \t H^i(0)) \\
 = & \curl \G_r^{} (-\frac{2|B|}{\tau}\e_3 \t \po) + \d k^2 \G_r^{\d \k} (-\frac{i}{k \tau}|B|\e_3 \t H^i(0)) \\
 = & -i \d k \frac{2|B|}{\tau} \G_r^{\d \k} (-d_3 \po^*) + \d k^2 \G_r^{\d \k} (-\frac{i}{k \tau}|B|\e_3 \t H^i(0)) \\
 = & -i \d k \frac{2|B|}{\tau} \G_r^{\d \k}  (-p_3 \di^*) = 0,
\end{align*}
where we have used the simple identity
$ 
    -\hi(0) + d' \t \ei(0) = 2d_3(p_2,-p_1,0).
$ 
In addition, we note that $d'\t \hi(0) + \ei(0) = 0$. Therefore,  the non-integral terms in (\ref{aux:firstorder:3}) and (\ref{aux:firstorder:5}) can be cancelled. Moreover, for any vector $\a \in \R^3$, we can check from \eqref{eq:proker} that
\begin{align} \label{aux:rela:1}
& \curl \G_r^{\d \k} \a  = i \d k \G_r^{\d \k} \di^* \t \a,  \q
\curl \G_r^{\d \k} \di^* \t \a = - i \d k \G_r^{\d \k} \a, 
\end{align}
and the vector identities
\begin{align}\label{aux:rela:3}
    &\di^* \t a' = - d_3 \e_3 \t \a + (d' \t  \a)_3 \e_3, \q
    (d' \t \a)' = \di^* \t (a_3 \e_3). 
\end{align}
Now recalling (\ref{aux:firstorder:2}) and (\ref{aux:firstorder:3}), and using  (\ref{aux:rela:1}) and (\ref{aux:rela:3}), we can deduce that
\begin{align*}
    &\d k^2 \G_r^{\d \k} \frac{1-\eec}{\tau d_3} ((d' \t \int_{\p B}\w{\y} \pn{u^h} d\sigma)_3 \e_3 - d_3\e_3 \t  \int_{\p B}\w{\y} \pn{u^h} d\sigma) \\
    = & \d k^2 \G^{\d \k}_r \frac{1-\eec}{\tau d_3} \di^* \t (\int_{\p B}\w{\y}\pn{u^h} d\sigma)'\\
    = & \frac{ik(\eec -1)}{\tau d_3} \curl \G_r^{\d \k} (\int_{\p B} \w{\y} \pn{u^h} d\sigma)'.
\end{align*}
Similarly, we can derive by means of (\ref{aux:rela:1}) and (\ref{aux:rela:3}) that
\begin{align*}
    \curl \G_r^{\d \k} \frac{1-\muc}{\tau d_3} \di^* \t (\int_{\p B} \w{\y} \pn{u^e} d\sigma)_3 \e_3 = \d k^2 \G_r^{\d \k} \frac{i}{\tau k_3} (\int_{\p B}\w{\y}(\muc - 1)\pn{u^e} d\sigma)_3 \e_3.
\end{align*}
Combining these observations above and substituting the expressions to (\ref{appro:reflection}),
we obtain the approximate scattered wave $\w{E}^r_p$ by adding up the electric dipole and magnetic dipole:
\begin{align} \label{for:reflection:final}
\w{E}^r_{p}(\w{\x}) = & \curl \G_r^{\d \k}(\w{\x}) {\bf J}'_{\zm}
+ \d k^2 \G^{\d \k}_r(\w{\x}) ({\bf J}_\ze)_3\e_3 + O(\d^2),
\end{align}
where ${\bf J}_\zm$ and ${\bf J}_\ze$ are defined by
\begin{align}
    {\bf J}_\zm:= & \frac{i}{\tau k_3} \int_{\p B} \w{\y} k^2 (\eec-1) \pn{u^h} d\sigma
    -\frac{i}{\tau k_3}\int_{\p B} \w{\y} (\lm + \K^{0,*}_\zm)^{-1}\frac{k_c^2 - k^2}{1 -\muc}\pn{u^h} d\sigma \nb\\
    & + \frac{1}{\tau d_3(\muc -1)} \int_{\p B} \w{\y} (\lm + \K_\zm^{0,*})^{-1}[\n \dd  \hi(0)]d\sigma  \label{eq:Jm}\\
     {\bf J}_\ze:= & \frac{i}{\tau k_3}\int_{\p B} (\muc -1) \w{\y} \pn{u^e} d\sigma - \frac{i}{\tau k_3} \int_{\p B}\frac{k_c^2 - k^2}{k^2(1-\eec)} \w{\y} (\lee +  \K^{0,*}_\ze)^{-1}\pn{u^e}d\sigma \nb\\
     & + \frac{i}{\tau k_3(1-\eec)} \int_{\p B} \w{\y} (\lee + \K^{0,*}_\ze)^{-1} [\n \dd \ei(0)]d\sigma. \label{eq:Je}
\end{align}

Next, we compute these two dipoles  ${\bf J}_\zm$ and  ${\bf J}_\ze$, respectively.
For ${\bf J}_\zm$, noting the relation
\begin{align}
     k^2(\eec - 1)(\lm + \K_\zm^{0,*}) + \frac{k^2 - k_c^2}{1 - \muc}
    &=  k^2(\eec - 1) ( \lm + \K_\zm^{0,*} + \frac{1 - \eec \muc}{(\eec -1 )(1-\muc)}) \notag \\
    &=  k^2(\eec - 1)(-\lee + \K_\zm^{0,*}), \label{aux:jm:coff}
\end{align}
we add the first two terms in ${\bf J}_\zm$ to obtain
\begin{align} \label{terms:mm:1}
    \frac{i k^2}{\tau k_3} \int_{\p B}(\eec - 1)\w{\y} (-\lee +  \K_\zm^{0,*})(\lm + \K_\zm^{0,*})^{-1} \pn{u^h} d\sigma.
\end{align}
But by applying Lemma \ref{lem:interpret:zeroorder} and the jump relation of the Neumann-Poincar\'{e} operator,
we get
\begin{align}
    &\pn{u^h} = \frac{i}{k\eec} \n \dd \hi(0)  + \frac{i}{k \eec (1-\eec)} (-\frac{1}{2} + \lee - \lee +  \K_\zm^{0,*}) (\lee - \K_m^{0,*})^{-1}[\n \dd \hi(0)]  \notag \\
    = & \frac{i}{k(\eec - 1)} \n \dd \hi(0) + \frac{i}{k(1-\eec)^2}(\lee - \K_\zm^{0,*})^{-1}[\n \dd \hi(0)].  \label{aux:jm:puh}
\end{align}
Then it follows from (\ref{terms:mm:1}) that
\begin{align}
    & \frac{i k^2}{\tau k_3} \int_{\p B}(\eec - 1) \w{\y} (-\lee +  \K_\zm^{0,*})(\lm + \K_\zm^{0,*})^{-1} \pn{u^h} d\sigma \notag \\
   = & \frac{i k(\eec - 1)}{\tau d_3} \int_{\p B}  \w{\y}  (-\lee +  \K_\zm^{0,*})(\lm + \K_\zm^{0,*})^{-1} \frac{i}{k(\eec - 1)} \n \dd \hi(0)d\sigma  \nb\\
    & +\frac{i k(\eec - 1)}{\tau d_3} \int_{\p B}  \w{\y} (-\lee +  \K_m^{0,*})(\lm +     \K_\zm^{0,*})^{-1} \frac{i}{k(1-\eec)^2}(\lee -\K_\zm^{0,*})^{-1}[\n \dd \hi(0)]d\sigma \nb \\
   = & \frac{1}{\tau d_3} \int_{\p B} \w{\y} (\lee + \lm) (\lm + \K^{0,*}_\zm)^{-1}[\n \dd \hi(0)] + \frac{1}{\tau d_3(\eec -1)} \int_{\p B} \w{\y} (\lm + \K^{0,*}_\zm)^{-1}[\n \dd \hi(0)] d\sigma.
\end{align}
Combining the above results, along with the relation
\begin{equation} \label{aux:rela:basic}
    -\lm - \lee + \frac{1}{1-\eec} + \frac{1}{1-\muc} = 1,
\end{equation}
we arrive at the desired expression
\begin{equation} \label{for:jm}
    {\bf J}_\zm = -\frac{1}{\tau d_3} \int_{\p B} \w{\y} (\lm + \K_\zm^{0,*})^{-1}[\n \dd \hi(0)]d\sigma\,.
\end{equation}

We now compute ${\bf J}_\ze$. Similarly to the results (\ref{aux:jm:coff}) and (\ref{aux:jm:puh}),
we have

\begin{equation}
    (\muc -1)(\lee + \K_\ze^{0,*}) + \frac{k_c^2 - k^2}{k^2(\eec - 1)}
=   (-\lm + \K^{0,*}_
\ze)(\muc - 1),
\end{equation}
and
\begin{align}
    & \pn{u^e} = \frac{1}{\muc}\n \dd \ei((0) + \frac{1}{\muc(1 - \muc)} (-\frac{1}{2}+\K^{0,*}_\ze)(\lm - \K^{0,*}_\ze)^{-1}[\n \dd \ei(0)]  \notag \\
   = &  \frac{1}{\muc - 1} \n \dd \ei(0) + \frac{1}{(1 - \muc)^2}(\lm - \K^{0,*}_\ze)^{-1}[\n \dd \ei(0)].
\end{align}
Applying these two expressions and (\ref{aux:rela:basic}) yields
\begin{align} \label{for:je}
    {\bf J}_\ze = \frac{i}{\tau k_3}\int_{\p B} \w{\y}(\lee + \K_\ze^{0,*})^{-1}[\n \dd \ei(0)]d\sigma.
\end{align}
Now,  by substituting (\ref{for:je}) and (\ref{for:jm}) into (\ref{for:reflection:final}),  and using the relation \eqref{aux:rela:1},
we come to the main result of this section.

\begin{theorem} \label{thm:mainresult}
When ${\d}/({d_\sigma d_\sigma^*})$ is sufficiently small, for $\x$ away from the thin layer $\D$, the scattered electric field $E^r$ has the asymptotic expression pointwisely as $\d \to 0$:
\begin{align} \label{eq:mainresult}
    E^r(\x) = \d k \G^\k_r(\x) (i \di^* \t {\bf J}'_\zm + k ({\bf J}_\ze)_3 \e_3 ) + O(\d^2),
\end{align}
where ${\bf J}_{\ze}$ and ${\bf J}_{\zm}$ are given by \eqref{for:je} and \eqref{for:jm}, respectively.
\end{theorem}

\begin{remark}
The geometry of the microstructure $D$ of the thin layer can be quite complicated, e.g.,
a domain with a hole or a domain with multiple connected components,
although we often assume it is simple connected with a connected boundary
for simplicity (e.g., Lemma\,\ref{thm:spec:no}). Therefore, our results are very general,
and are still true for the strongly coupled multi-layer case,
i.e., there are multiple layers of close-to-touching nanoparticles.
\end{remark}

From Theorem\,\ref{thm:mainresult}, we can clearly see the anomalous electromagnetic scattering is due to the occurrence of the mixed collective plasmonic resonances. To make it more precise,
let us define the electric and magnetic polarization tensors:
\begin{equation*}
    M_{\ze}(\lee,B) = \int_{\p B}\w{\y}(\lee + \K_\ze^{0,*})^{-1}[\n] d \sigma, \quad
    M_{\zm}(\lm,B) = \int_{\p B}\w{\y}(\lm + \K_\zm^{0,*})^{-1}[\n] d \sigma.
\end{equation*}
By the definition of $\G^\k_r(\x)$ \eqref{eq:proker}, and with the help of projections $\e_3 \otimes \e_3$ and $\I - \e_3 \otimes \e_3$ and the relations
\begin{equation*}
    \ei(0) = - 2 \e_3 \otimes \e_3 \po^*\,, \q \hi(0) = -2 (\I - \e_3 \otimes \e_3)\di^* \t \po^*,
\end{equation*}  we can reformulate \eqref{eq:mainresult} in a more compact form:
\begin{align}
    E^r(\x) & = \d k (\I - \di^* \otimes \di^*)e^{i\k^*\dd \x} (i \di^* \t {\bf J}'_\zm +  k ({\bf J}_\ze)_3 \e_3) +O(\d^2)
     =  \frac{2 i \d k }{\tau d_3} e^{i\k^*\dd \x} \mathscr{R} \po^*+O(\d^2),
     \label{eq:maincompacfor}
\end{align}
where the reflection scattering matrix $\mathscr{R}$ is given by
\begin{equation*}
    \mathscr{R} = (\I - \di^* \otimes \di^*)(\di^*\t (1-\e_3 \otimes \e_3) M_{\zm}(\lm,B)(1-\e_3 \otimes \e_3)\di^* \t \I   -\e_3 \otimes \e_3 M_{\ze}(\lee,B) \e_3 \otimes \e_3).
\end{equation*}
We emphasize that $\mathscr{R}$ as a three by three matrix should be regarded as a linear mapping defined on the two dimensional subspace of $\R^3$ perpendicular to $\di^*$, which chacterizes the polarization conversion. In the traditional optical systems, the scattering effect of such kind of subwavelength rough surface is basically negligible so that $\mathscr{R}$ plays a limited role. However, due to the large negative permittivity and permeability of the plasmonic nanoparticles \cite{jain2006calculated,sarid2010modern}, $\lm(\omega)$ and $\lee(\omega)$ can approach the spectrum of $-\K^{0,*}_\zm$ and $-\K^{0,*}_{\ze}$ such that the elements in $M_{\ze}(\lee,B)$ and $ M_{\zm}(\lm,B)$ may blow up with an enhancement order
${1}/{d_\sigma^*}$.
Therefore, following \cite{ammari2016surface,ammari2016plasmaxwell,ammari2017mathematicalscalar}, we may define the collective plasmonic resonances by the frequencies $\omega$ satisfying
\begin{equation*}
    d(\lee(\omega),-\K^{0,*}_{\ze}) \ll 1 \ \text{or}\   d(\lm(\omega),-\K^{0,*}_{\zm}) \ll 1.
\end{equation*}
It is worth emphasizing that these frequencies generally are very different from the single particle case. Physically, these periodically distributed plasmonic nanoparticles can resonate as a whole so that a nanoscale thin layer can significantly affect the wave propagation at the macroscale.
We refer the readers to \cite{ammari2016mathematical}
for some numerical evidences on collective plasmonic resonances.  If the collective plasmonic resonances are excited, the effect of the reflection scattering matrix $\mathscr{R}$ can overcome the size parameter $\d$ and become
visible, giving the possibility of achieving a desired far field pattern.
 However, our electromagnetic plasmonic metasurface, as all the nano-optic devices, still faces many fundamental limits. Actually, following \cite{arbabi2017fundamental}, we may decompose $E^r(\x)$ into two plane waves with orthogonal polarizations: one with polarization $\po^*$ and the other with a polarization orthogonal to $\po^*$.
Moreover, we can introduce the reflection coefficients and polarization conversion coefficients to measure the functionalities of the metasurface, and then analyze their bounds and fundamental relations via holomorphic functional calculus\cite{ammari2017mathematicalscalar}.

\subsection{Equivalent impedance boundary condition}
The final goal of this work is to present an impedance boundary condition approximation. First, we recall the definition of the surface scalar curl and surface vector curl. In fact, they have the explicit forms on the reflective plane $\Gamma$: $\cst u = \frac{\p u_2}{\p x_1} - \frac{\p u_1}{\p x_2}$ for vector function ${\bf u} = (u_1,u_2,0)$ and $\cvt v = (\frac{\p v}{\p x_2},-\frac{\p v}{\p x_1},0)$ for scalar function $v$. We first consider a simple case: the plasmonic nanoparticle is non-magnetic, i.e., $\muc = 1$.
In this case, Theorem \ref{thm:mainresult} indicates that in the far field, the total electric field can be approximated by $E^\d$:
\begin{equation*}
    E^\d := E^i + \d k^2  \G_r^\k ({\bf J}_e)_3 \e_3.
\end{equation*}
Introduce
\begin{equation*}
    \beta_\ze := -\frac{1}{\tau} \int_{\p B} y_3 (\lee + \K_\ze^{0,*})^{-1}[v_3]d\sigma.
\end{equation*}
Then a simple calculation gives, with the help of \eqref{eq:proker} and $\e_3 \t E^i|_{\Gamma} = 0$,
\begin{align*}
    \e_3 \t E^\d|_{\Gamma} & = - \d k^2 \e_3 \t \di^* d_3 ({\bf J}_\ze)_3  e^{ik d' \cdot x'} \\
    & = i \d k 2p_3 \e_3 \t \di^*\beta_\ze  e^{ik d' \cdot x'} \\
    & = \d e^{i k d'\dd x'}(ik d_2,-ik d_1,0)\beta_\ze 2p_3.
\end{align*}
Hence we can derive, by noting that $\cst (\hi)'|_{\Gamma} = -ik \e_3 \dd E^i|_{\Gamma}$, when ${\d}/({d_\sigma d^*_\sigma}) \to 0$:
\begin{align*}
    \e_3 \t E^\d|_{\Gamma} &=  \d \beta_\ze  \cvt \e_3 \dd E^i|_\Gamma
     = \d \beta_\ze \cvt \frac{i}{k} \cst (\hi)'|_{\Gamma}
    = \d \frac{i\beta_\ze}{k} \cvt \cst (H^\d)'|_{\Gamma} + O(\d^2).
\end{align*}
This yields the equivalent impedance boundary condition
\begin{equation} \label{eq:bounappnm}
    \e_3 \t E^\d|_{\Gamma} =  \d \frac{i\beta_\ze}{k} \cvt \cst (H^\d)'|_{\Gamma}
\end{equation}
to approximate the effect of the thin layer in the far field,
up to the second order term. Moreover, this is uniformly valid with respect to the resonance.



 We now consider the magnetic plasmonic nanoparticle, i.e., $\muc \ne 1$, and
introduce the $2\times 2$ matrix
\begin{equation*}
    D_\zm = \frac{1}{\tau} \int_{\p B}y'(\lm + \K_\zm^{0,*})^{-1}[\n'] d \sigma.
\end{equation*}
According to Theorem \ref{thm:mainresult}, the electric field can be approximated by
\begin{equation*}
    E^\d = E^i + \d k \G^\k_r(\x) (i \di^* \t {\bf J}'_\zm + k ({\bf J}_\ze)_3 \e_3 ).
\end{equation*}
In a similar way as in the  non-magnetic case, we can find that
\begin{equation*}
     \e_3 \t E^\d|_{\Gamma} =  \d \frac{i\beta_\ze}{k} \cvt \cst (H^\d)'|_{\Gamma} - i k \d D_\zm (H^ \d)'|_{\Gamma}  +   O(\d^2),
\end{equation*}
with the help of the following observation:
\begin{equation*}
    i\d k \e_3 \t (\G_r^\k \di^* \t {\bf J}'_\zm)|_\Gamma = i \d k \e_3 \t (\di^* \t {\bf J}_\zm')e^{ik d' \cdot x'}|_\Gamma = - i k \d D_\zm (\hi)'|_{\Gamma}.
\end{equation*}
This yields the following effective impedance boundary condition
\begin{equation} \label{boundary approximation}
    \e_3 \t E^\d|_{\Gamma} =  \d \frac{i\beta_\ze}{k} \cvt \cst (H^\d)'|_{\Gamma}  -i k \d D_\zm (H^ \d)'|_{\Gamma}
\end{equation}
to approximate the effect of the thin layer in the macroscopic scale,
up to the second order term. And this is again uniformly valid with respect to the resonance.

\section{Concluding remarks and extensions} \label{concluding}
In this work, we have studied the scattering effect of the periodically distributed plasmonic nanoparticles in the homogenization regime. For the subwavelength structures of such patterns,
 a Leontovich boundary condition (cf. (\ref{boundary approximation})) has been derived
for the approximation of the scattered field in both magnetic and non-magnetic cases.
A similar problem setting was considered
in \cite{delourme2013well,delourme2015high,delourme:tel-00650354}, where the thin layer was
made of dielectric particles, for which the standard variational approach applies. However, the variational framework breaks down in the resonant case, hence instead we have adopted the layer potential theories
in this work to analyze the singularity and prove the uniform validation of the boundary condition approximations.
Our results provide a relatively complete picture of the mechanism for the electromagnetic plasmonic metasurfaces
and can be easily modified to cope with other regimes and boundary conditions.
Therefore this work may be viewed as a generalization of the standard homogenization theory to resonant micro-structures.
And our theoretical analysis and findings may help design a metasurface that can resonate at some specific dense set of frequencies to further realize the broadband wave modulation. In addition, it is also a very interesting and challenging topic to understand how to reconstruct fine structures of thin layers in terms of the scattered field under resonance.

 Although we only consider the homogenization regime in this work since it is the most interesting and important case where the collective resonance can happen,
our results and analysis in this work can actually be extended to several important physical regimes and applications.
First, our approach can be directly applied to other important regimes, such as
$$
\mbox{size of particle $\ll$ period $\sim$ wave length, \quad or \quad size of particle $\ll$ period $\ll$ wave length.}
$$
However, we may not expect the collective plasmonic resonances
in these configurations, since the particles are well separated in some sense though they are distributed in a certain pattern. In fact, the scattering field will be locally dominated by the resonance modes excited by a single nanoparticle. Therefore, the thin layers under these regimes may not have the capability to realize the control of the electromagnetic wave
in the macroscopic scale.
In this work, we have considered only the perfect conducting boundary conditions on the bottom
surface $\Gamma$, but our results and analysis can be extended to other boundary conditions as well,
by replacing the Green's tensors defined in the Section \ref{layerpotential} by the ones
satisfying other specified boundary conditions.
As we have mentioned earlier, our results remain the same for the multiple close-to-touching thin layers.
In fact,  the generalization to the well-separated multi-layer case, i.e.,
$$
\mbox{size of particle $\sim$ period $\ll$ distance between two layers $\sim$ wavelength $\sim$ 1,}
$$
is also direct since the scattering effect of each layer can be considered independently due to the weak interaction.
Formally, suppose we have $n$ thin layers associated with the approximate scattered waves $E^r_1$, $\cdots$,
$E^r_n$ given by similar terms to \eqref{eq:mainresult}, then for this multi-layer structure, the total approximate scattered wave $E^r_{app}$ can be written as $E^r_{app} = E^r_1+\cdots+E^r_n$. With these design flexibilities and extension remarks, our theoretical findings shed also light on the mathematical understanding of electromagnetic plasmonic metasurfaces
and their related optimal design problems.

\end{document}